\documentclass[10pt]{amsart}
\usepackage{amsmath, amsthm, amssymb}
\usepackage{graphicx}
\usepackage{mathbbol}
\usepackage{txfonts}
\usepackage[usenames]{color}

\newtheorem{thm}{Theorem}[section]

\newtheorem{ex}[thm]{Example}
\newtheorem{lem}[thm]{Lemma}   
\newtheorem{prop}[thm]{Proposition}
\newtheorem{defn}[thm]{Definition}
\newtheorem{rmrk}[thm]{Remark}   
\newtheorem{question}[thm]{Open Question}   

\newcommand{\be}{\begin{equation}}
\newcommand{\ee}{\end{equation}}
\newcommand{\bee}{\begin{equation*}}
\newcommand{\eee}{\end{equation*}}

\newcommand{\N}{\mathbb{N}}

\newcommand{\R}{\mathbb{R}}

\newcommand{\E}{\mathbb{E}}
\newcommand{\HH}{\mathbb{H}}
\newcommand{\Z}{\mathbb{Z}}

\newcommand{\vare}{\varepsilon}
\newcommand{\union}{\cup}
\newcommand{\disjointunion}{\sqcup}

\newcommand{\diam}{\operatorname{Diam}}

\newcommand{\vol}{\operatorname{Vol}}

\newcommand{\Hto}{\stackrel { \textrm{H}}{\longrightarrow} }
\newcommand{\GHto}{\stackrel { \textrm{GH}}{\longrightarrow} }

\begin{document}

\title{Sequences of Open Riemannian Manifolds with Boundary}

\author{Raquel Perales}
\thanks{The first author is a doctoral student at Stony Brook.}
\address{SUNY at Stony Brook}
\email{praquel@math.sunysb.edu}

\author{Christina Sormani}
\thanks{The second author's research is partially supported by 
NSF DMS 10060059.}
\address{CUNY Graduate Center and Lehman College}
\email{sormanic@member.ams.org}

\keywords{}
\maketitle

\begin{abstract}
We consider sequences of open Riemannian manifolds
with boundary that have no regularity conditions on the boundary.  To
define a reasonable notion of a limit of such a sequence, 
we examine ``$\delta$ inner regions" which avoid the 
boundary by a distance $\delta$.   We prove Gromov-Hausdorff
compactness theorems for sequences of these
``$\delta$ inner regions".  We then build ``glued limit spaces"
out of the Gromov-Hausdorff limits of these $\delta$ interior regions 
and study the properties of these glued limit spaces. 
Our main applications assume the sequence is noncollapsing and
has nonnegative Ricci curvature.   We
include open questions.
\end{abstract}

\section{Introduction}

Recall that Gromov's Ricci Compactness Theorem states that a
sequence of compact
Riemannian manifolds with nonnegative Ricci curvature
and a uniform upper bound on diameter has a subsequence which
converges in the Gromov-Hausdorff sense to a metric space
\cite{Gromov-metric}.   When the sequence of manifolds
is noncollapsing, then the Gromov-Hausdorff limit spaces have
a variety of properties, particularly restrictions on their metrics, 
their Hausdorff measures, and their topologies.  These properties 
were proven by 
Cheeger, Colding, Naber, Wei and the second author
(c.f. \cite{ChCo-PartI}, 
\cite{Colding-volume}, \cite{Colding-Naber-2} and \cite{SorWei1}).

Here we consider an open Riemannian manifold, $(M^m,g)$,
endowed with the length metric, $d_M$,
as in (\ref{length-metric}).  We
define the boundary to be
\be \label{defn-bndry}
\partial M = \bar{M}\setminus M
\ee
where $\bar{M}$ is the metric completion of $M$. For example,
$(M^m, g)$ may be a smooth manifold with boundary.  However,
we do not require any smoothness conditions on this boundary.

First observe that Gromov's Ricci Compactness Theorem
does not hold for precompact open manifolds with boundary
that have a uniform upper bound on diameter even if they are
flat and two dimensional:

\begin{ex} \label{gold-foils}
The $j$-fold covering spaces, $M_j$, of the 
annulus, $Ann_{0}(1/j,1)\subset \E^2$,
depicted in Figure~\ref{fig-golden-foils},
are flat surfaces such that
\be
\diam(M_j)\le 2+\pi \textrm{ and } \vol(M_j)= j (\pi - \pi (1/j)^2).
\ee
See Remark~\ref{rmrk-gold-foils} for the proof that
there is no subsequence of these
spaces with a Gromov-Hausdorff limit.   
\end{ex}

\begin{figure}[htbp] 
   \centering
   \includegraphics[width=3.5in]{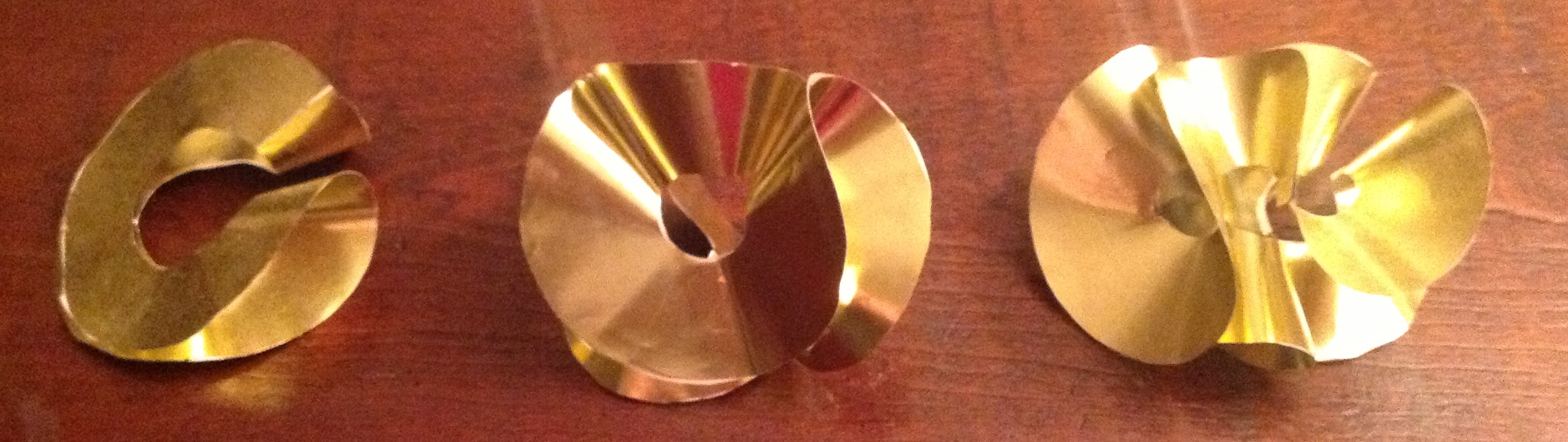} 
 \caption{Models of Example~\ref{gold-foils}: $M_2$, $M_3$, $M_4$...}
   \label{fig-golden-foils}
\end{figure}

Assuming both a uniform upper bound on volume and diameter,
we still do not have Gromov-Hausdorff compactness:

\begin{ex} \label{many-splines}
The smooth regions, $M_j\subset \E^2$, 
with many splines 
depicted in Figure~\ref{fig-many-splines}
have no subsequence with a Gromov-Hausdorff limit.  See Example~\ref{ex-many-splines} for details.   
\end{ex}

\begin{figure}[htbp] 
   \centering
   \includegraphics[width=3.5in]{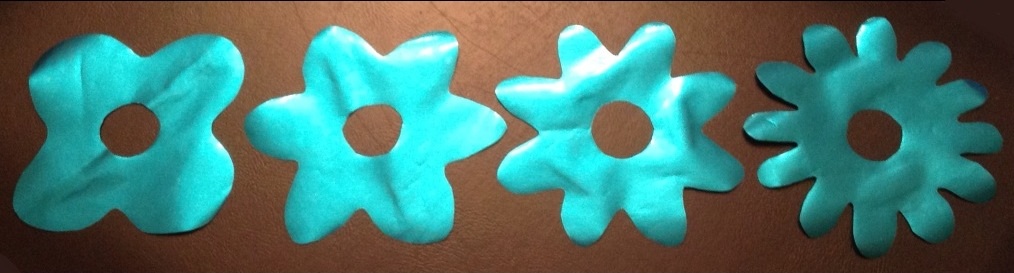} 
 \caption{Models of Example~\ref{many-splines}: $M_4$, $M_6$, $M_8$, $M_{12}$...}
   \label{fig-many-splines}
\end{figure}

Kodani \cite{Kodani-1990}, Anderson-Katsuda-Kurylev-Lussas-Taylor \cite{AKKLT}, Wong \cite{Wong-2008} and, most recently, Knox \cite{Knox}
have proven compactness theorems for sequences of Riemannian
manifolds with boundary assuming curvature controls on the boundary.   A survey of these results has been written by the first author \cite{Perales-survey}.
Since we do not wish to assume the boundary is smooth,
we prove compactness theorems
for regions which avoid the boundary [Theorem~\ref{main-thm}].
We then glue together the limits of these regions
[Theorem~\ref{def-glue}]
and prove these glued limit spaces have nice properties
[Theorem~\ref{thm-ricci-glued-lim-measure}].   

\begin{defn} \label{defn-1}
Given an open Riemannian manifold, $(M, g_M)$, and a positive 
$\delta>0$, we define the $\delta$ inner region as follows
\be
M^\delta=\Big\{ x\in M: \,\, d_M(x, \partial M) > \delta\Big\}
\ee
where $\partial M$ is defined as in (\ref{defn-bndry}),
\be \label{length-metric}
d_M(x,y) := \inf\Big\{ L_g(C): \,\, C:[0,1]\to M, \, C(0)=x, \, C(1)=y \Big\}
\ee
and 
\be
L_g(C)=\int_0^1 g(C'(t),C'(t))\, dt.
\ee
\end{defn}

Note that there are two metrics 
on the $\delta$ interior region, $M^\delta$: the restricted metric, $d_M$,
and the induced length metric, 
\be \label{intrinsic-distance}
d_{M^\delta} (x,y) :=\inf\Big\{ L_g(C): \,\, C:[0,1]\to M^\delta, \, C(0)=x, \, C(1)=y \Big\}.
\ee
Note that $d_{M^\delta}$ is only defined between points in
connected components of $M^\delta$.   The intrinsic diameter
\be \label{intrinsic-delta-diameter}
\diam(M^\delta, d_{M^\delta})=\sup\left\{d_{M^\delta}(x,y): \, x,y\in M^\delta\right\},
\ee
will be infinite if $M^\delta$ is not connected by rectifiable paths.

\begin{thm}\label{main-thm}
Given $m\in \N$, $\delta>0$,  $D>0$, $V>0$, and $\theta>0$, set
$
\mathcal{M}^{m, \delta, D, V}_{\theta}  
$
to be the class of open Riemannian manifolds, $M$, with boundary
of dimension $m$,
with nonnegative Ricci curvature, $\vol(M)\le V$, and
\be \label{diam-here}
\diam(M^\delta,{d_{M^\delta}})
\le D,
\ee
that are noncollapsing at a point,
\be \label{V-epsilon}
\exists q\in M^\delta
\textrm{ such that } \vol(B_q(\delta)) \ge \theta\delta^m.
\ee
If $(M_j, g_j) \subset \mathcal{M}^{m, \delta, D, V}_{\theta}$, then
there is a subsequence, $(M^\delta_{j_k}, d_{M_{j_k}})$, such that the
metric completions
with the restricted metric, $d_{M_j}$,
converge in the Gromov-Hausdorff sense to a metric space,
$(Y^\delta,d)$.  
\end{thm}  

Example~\ref{many-splines} satisfies the conditions of this theorem,
demonstrating why we can only obtain Gromov-Hausdorff convergence
of the $M_j^\delta$ instead of the $M_j$ themselves.   The $M_j^\delta$
of Example~\ref{gold-foils} do not have Gromov-Hausdorff converging
subsequences (see Remark~\ref{rmrk-gold-foils}) demonstrating the
necessity of the hypothesis requiring an upper volume bound.   
In Theorem~\ref{sc-const}, stated within, we remove the intrinsic
diameter condition, (\ref{diam-here}), and the noncollapsing condition,
(\ref{V-epsilon}), and assume conditions on closed geodesics and
constant sectional curvature instead.

Theorem~\ref{main-thm} and Theorem~\ref{sc-const} are proven in
Section~\ref{sect-main}.  First we review of Gromov-Hausdorff
convergence in Section~\ref{sect-background}.
In Sections~\ref{sect-inner} and~\ref{sect-manifolds} we study the
limits of inner regions in sequences of manifolds that have Gromov-Hausdorff
limits.  See in particular Theorem~\ref{GH-lim-gives-delta-lim}.  
These sections contain many examples.

In Section~\ref{sect-glued-limit-spaces} we define glued limit
spaces for any sequence of open Riemannian manifolds,
$(M_j,g_j)$ assuming that for all $\delta>0$, $(M_j^\delta, d_j)$
converge in the Gromov-Hausdorff sense to a metric space
$(Y^\delta, d_\delta)$.  We build a ``glued limit space", $(Y, d_Y)$, 
from these
$Y^\delta$ in Theorem~\ref{thm-gluing-isoms} and Theorem~\ref{def-glue}.
The metric completion of a glued limit space is called a ``completed glued limit space".

Note that this glued limit space may exist even when $(M_j, d_j)$
has no Gromov-Hausdorff limit as in Examples~\ref{ex-many-splines},
[Remark~\ref{rmrk-many-splines}].   
The glued limit may not be precompact even 
when one has a sequence of flat Riemannian manifolds with boundary
[Examples~\ref{ex-decreasing-splines} and~\ref{ex-many-pages-2}].     

In general the completed glued limit space of a sequence of $M_j$ need
not be unique [Example~\ref{ex-nonunique}].  However, if the
$(M_j, d_{M_j})$ 
have a Gromov-Hausdorff limit, $(X, d_X)$, then the
completed glued limit space is unique and is embedded isometrically into 
$X$ [Theorem~\ref{thm-uniquenessGluedLimit}].   The completed glued limit
space need not be isometric to the Gromov-Hausdorff limit [Example~\ref{ex-one-spline}] even when the $(M_j, g_j)$ are regions in the Euclidean
plane satisfying all the hypothesis of Theorem~\ref{main-thm}
[Remark~\ref{rmrk-one-spline-2}].   
Intuitively, regions which collapse relative to the boundary disappear
while regions which collapse that lie far from the boundary, need not
disappear.

In Section~\ref{sect-curv} we apply Theorems~\ref{sc-const} 
and~\ref{main-thm} to construct glued limit spaces for sequences
of manifolds with curvature bounds [Theorems~\ref{glued-constsec-limits}
and~\ref{glued-Ricci-limits}.   In Section~\ref{sect-prop} we explore
the properties of these glued limit spaces.  First we present an example
where the curvature bounds in the sequence of manifolds is lost in the
Gromov-Hausdorff
limit [Example~\ref{ex-to-negative}].   Then we prove 
Proposition~\ref{sc-prop} concerning glued limits of manifolds with constant sectional
curvature.  We close with Theorem~\ref{thm-ricci-glued-lim-measure}, proving that glued limits constructed under the
conditions of Theorem~\ref{main-thm} have 
Hausdorff dimension $m$, Hausdorff measure $\le V$, and positive
density everywhere.   This final theorem is proven using Theorem~\ref{thm-balls-in-limits-1} which proves certain balls in glued limit spaces are the
Gromov-Hausdorff limits of nice balls in the open manifolds, combined
with the Bishop-Gromov Volume Comparison Theorem \cite{Gromov-metric}
and Colding's Volume Convergence Theorem \cite{Colding-volume}.

Throughout the paper we state open questions: Question~\ref{q-local-geod}, Question~\ref{q-sc-manifold},
Question~\ref{q-sc-unique}, Question~\ref{q-ricci-rectifiable}, 
and Question~\ref{q-ricci-unique}.   The first author is in the 
process of proving Question~\ref{q-ricci-rectifiable} as part of her doctoral
dissertation.  Please contact us if you would like to work on one of the
other open questions or if you are interested in extending
our theorems to the setting where the sequence has a 
negative uniform lower Ricci curvature bound or is allowed to collapse.

We would like to thank Stephanie Alexander (UIUC) for informing us about
the work of Wong and Kodani when we first began to
explore the question.   We'd like to thank 
Frank Morgan (Williams) and David Johnson (Lehigh) for their interest
and encouragement.   We'd like to thank Pedro Sol\'orzano (UC Riverside)
for looking over some of the proofs.    Thanks to Tabitha (IS 25), Penelope (IS 25) 
and Kendall (PS 32) for building the
models of Examples~\ref{gold-foils}-\ref{many-splines} depicted in Figures~\ref{fig-golden-foils}, \ref{fig-many-splines}
and~\ref{fig-many-splines-2} and for computing the
areas of the manifolds in Example~\ref{gold-foils} as part of a K-12 outreach.
Finally we would like to thank
Jorge Basilio (CUNY), Christine Briener (MIT), Maria Hempel (ETH Zurich),
Sajjad Lakzian (CUNY), Christopher Lonke (Carnegie Mellon),
Mike Munn (U Missouri at Columbia), Jacobus Portegies (Courant, NYU), and Timothy Susse (CUNY) for actively
participating with us in the CUNY Metric Geometry Reading Seminar in
the Summer of 2012 and Kenneth Knox (Stony Brook) for joining us in the Spring of 2013.

\section{Background} \label{sect-background}

Here we review Gromov-Hausdorff convergence and Gromov's
Compactness Theorem \cite{Gromov-metric}.  A good resource
for this material is \cite{BBI}.

\subsection{Hausdorff Convergence}

In \cite{Gromov-metric}, Gromov defined the Gromov-Hausdorff
distance between pairs of compact
metric spaces.  We review this definition here.

\begin{defn} [Hausdorff]
The Hausdorff distance between two compact subsets, $A_1, A_2$,
of a metric space, $Z$, with metric, $d_Z$, is defined
\be
d_H^Z(A_1, A_2) = \inf \Big\{ r: \, A_1 \subset T_r(A_2), \, A_2 \subset T_r(A_1)\Big\}
\ee
where the tubular neighborhood, $T_r(A)=\big\{ x\in Z: \, d_Z(x, A)<r\big\}$.
\end{defn}

Observe that if one has a sequence of compact subsets $A_j \subset Z$
such that $d_H(A_j , A_\infty) \to 0$, then for all $a\in A_\infty$
there exists $a_j \in A_j$ such that $\lim_{j\to \infty} a_j =a$.
One also has the following lemma:

\begin{lem} \label{lem-H-balls}
Suppose $A_j \subset Z$ are compact, $d_H^Z(A_j, A_\infty)=h_j \to 0$ and $a_j \in A_j$ such that
$d_Z(a_j, a_\infty)=\delta_j \to 0$.    Then for all $r>0$ there 
exists $r_j =r+\delta_j+h_j \to r$ such that the closed balls converge
\be
d_H^Z\Big(\bar{B}_{a_j}(r_j)\cap A_j, \bar{B}_{a_\infty}(r)\cap A_\infty\Big) \to 0.
\ee
\end{lem}

Here we are not assuming $A_\infty$ or $A_j$ are length spaces.
For completeness of exposition we include the proof of this well known
lemma:

\begin{proof} 
Suppose $x \in \bar{B}_{a_\infty}(r) \cap A_\infty$, then $d_Z(x, a_\infty)\leq r$
and $x\in A_\infty \subset T_{h_j}(A_j)$.  So there exists 
$y_j\in A_j$ such that $d_Z(x,y_j) < h_j$.  By triangle inequality,
\be
d(y_j, a_j) \le d(y_j,x)+d(x, a_\infty) + d(a_\infty, a_j)\le h_j + r + \delta_j =r_j.
\ee
Thus 
\be \label{first-part}
\bar{B}_{a_\infty}(r) \cap A_\infty \subset T_{h_j}(\bar{B}_{a_j}(r_j) \cap A_j).
\ee

Now we need only show there exists $\vare_j \to 0$ such that
\be
\bar{B}_{a_j}(r_j) \cap A_j \subset T_{\vare_j}(\bar{B}_{a_\infty}(r) \cap A_\infty).
\ee
Suppose not.   Then there exists $\vare_0>0$ such that
for all $j$ sufficiently large, there is an 
\be
x_j \in \Big(\bar{B}_{a_j}(r_j) \cap A_j \Big) \setminus T_{\vare_0}\Big(\bar{B}_{a_\infty}(r) \cap A_\infty\Big).
\ee
Since $Z$ is compact
and $T_{\vare_0}(\bar{B}_{a_\infty}(r) \cap A_\infty)$ is open,
 a subsequence of the $x_j$ converge to
some 
\be
x_\infty \notin T_{\vare_0}(\bar{B}_{a_\infty}(r) \cap A_\infty).
\ee
Since $d(x_j, a_j)\le r_j$, we have $d(x_\infty, a_\infty)\le r$.
Since $x_j \in A_j$, there exists $y_j \in A_\infty$ such that
$d(x_j, y_j) < h_j$.   By the triangle inequality
\be
y_j \in B_{a_\infty}(r+h_j) \cap A_\infty.
\ee 
Observe that for our subsequence $y_j \to x_\infty$, thus
\be
x_\infty \in \bar{B}_{a_\infty}(r)\cap A_\infty
\subset T_{\vare_0}(\bar{B}_{a_\infty}(r) \cap A_\infty)
\ee
which is a contradiction.
\end{proof}

\subsection{Gromov-Hausdorff Convergence}

 \begin{defn}
 An isometric embedding, $\varphi: (X, d_X) \to (Z, d_Z)$ between
 metric spaces is a mapping which preserves distances:
 \be
 d_Z( \varphi(x_1), \varphi(x_2)) = d_X(x_1, x_2)
 \ee
 \end{defn}

\begin{defn} [Gromov]
The Gromov-Hausdorff distance between a pair of compact metric
spaces, $(X_1, d_{X_1})$ and $(X_2, d_{X_2})$ is defined
\be
d_{GH}\Big(\big(X_1, d_{X_1}\big),\big(X_2, d_{X_2}\big)\Big)= 
\inf\Big\{ d_Z\big(\varphi_1(X_1), \varphi_2(X_2)\big): \,\, \varphi_i: X_i\to Z \,\Big\}
\ee
where the infimum is taken over all isometric embeddings
$\varphi_i: X_i \to Z$ and all metric spaces, $Z$.
\end{defn}

Gromov proved that the Gromov-Hausdorff distance is a distance
on the space of compact metric spaces.  When studying metric
spaces, $X_i$, which are only precompact, one takes the metric 
completions, $\bar{X}_i$,
before comparing such spaces using the Gromov-Hausdorff distance:

 \begin{defn}
 Given a precompact metric space space, $(X, d_X)$, the 
 metric completion, $(\bar{X}, d_X)$, consists of equivalence classes
 of Cauchy sequences, $\{x_1, x_2, x_3,...\}$, in $X$, where 
 \be
 d_X(\{x_j\}, \{y_j\}) = \lim_{j\to \infty} d_X(x_j, y_j)
 \ee
 and two Cauchy sequences are equivalent if the distance between them
 is $0$.   There is an isometric embedding
 \be
 \varphi: X \to \bar{X} \textrm{ such that } \varphi(x)=\{x,x,x,...\}.
 \ee
 \end{defn}
 
 In this paper we define the boundary of an open metric space 
 \be
 \partial X = \bar{X} \setminus X.
 \ee
 When $M$ is a smooth Riemannian manifold with boundary,
 then this notion of boundary agrees with the standard notion
 of boundary.   However, if $M$ is a smooth Riemannian manifold
 with a singular point removed, then the boundary in our setting
 is just the missing singular point.

\subsection{Lattices and Gromov-Hausdorff Convergence}

One technique that can be applied to produce amazingly
complicated Gromov-Hausdorff limits from surfaces, is to construct
lattices.   The basic well known lemma is as follows:

\begin{lem} \label{lem-lattice-1}
Let $X=[a_1,b_1]\times \cdots \times [a_k, b_k]$ with the taxi product
metric
\be
d_X\big( (x_1,...x_k), (y_1,...y_k) \big) = \sum_{i=1}^k |x_i-y_i|.
\ee
Then for any $\vare>0$ there exists a 2 dimensional manifold
$M_\vare$ such that
\be
d_{GH}(M_\vare, X) < \vare.
\ee
\end{lem}

The classic application of this lemma is to construct a Gromov-Hausdorff
limit of Riemannian surfaces which is infinite dimensional:

\begin{ex}
Let $X_j=[0,1]\times[0,1/2]\times \cdots \times [0, 1/2^j]$ with the
taxi metric and let 
\be
X=[0,1]\times[0,1/2]\times \cdots \times [0, 1/2^j]\times \cdots
\ee
be the infinite dimensional space also with the taxi metric:
\be
d_X\big( (x_1,x_2,...), (y_1,y_2,...) \big) = \sum_{i=1}^\infty |x_i-y_i|.
\ee
Then 
\be
d_{GH}(X_k, X) \le \sum_{j=k+1}^\infty 1/2^j=1/2^k \to 0.
\ee
Thus by Lemma~\ref{lem-lattice-1} we have a sequence of
surfaces $M_k$ converging to $X$ as well.
\end{ex}

Since we are interested in manifolds with boundary, we will
prove a stronger version of Lemma~\ref{lem-lattice-1} that can
be applied to produce examples later in the paper.

\begin{prop} \label{prop-lattice}
Suppose $X=[a_1,b_1]\times \cdots \times [a_k, b_k]$
with the taxi product metric and $A \subset \partial X$ (possibly empty), 
then for any $\vare>0$ there exists an open Riemannian surface, $M$, 
with boundary, $\partial M$ (possible empty) such that
\be
d_{GH}(M, X) < \vare \textrm{ and } d_{GH}(\partial M, A) < \vare.
\ee 
We can also prove that if we have a collection of $X_k$ and
$A_k \subset \partial X_k$ as above with subsets $B_k \subset X_k$ and isometric
embeddings $\psi_k: B_{k+1}\to B_{k}$, and we glue together
$X=X_1 \,\disjointunion \, X_2 \,\disjointunion \, \cdots \,\disjointunion \, X_k$ along these isometric embeddings,
and set $A=\union A_k \subset X$, then for any $\vare>0$
we have an open Riemannian surface, $M$, 
with boundary, $\partial M$ (possible empty) such that
\be
d_{GH}(M, X) < \vare \textrm{ and } d_{GH}(\partial M, A) < \vare.
\ee 
In fact, for any $\delta>0$, using the restricted distances, we have
\be
d_{GH}\big( (M\setminus T_\delta(\partial M), d_M),
(X\setminus T_\delta(A), d_X) \big) < \vare.
\ee
\end{prop}

\begin{proof}
For the first part, we take a lattice $Y'_\vare\subset Y_\vare \subset X$
such that $X \subset T_{\vare/2}(Y_\vare)$.   Here we use $Y'_\vare$
to denote the points and $Y_\vare$ to include $1$ dimensional edges
between the points in the lattice.   Observe
that $d_{Y_\vare}(y_1, y_2)=d_X(y_1, y_2)$ because we
are using the taxi norm.   Let $A_\vare \subset Y'_\vare$ be chosen
such that $A_\vare \subset T_{\vare/2}(A)$.   So
\be
d_{GH}(Y_\vare, X) < \vare/2 \textrm{ and } d_{GH}(A_\vare, A) < \vare/2.
\ee 
Note that we may now view $Y_\vare$ as a graph.   For example,
if $X=[0,5]\times[0,6]$ and $A=[0,5]\times \{6\}$ and $\epsilon=1$,
then the left side of
Figure~\ref{fig-taxi-lattice} is the graph $Y_\vare$ with 
$A_\vare$ is depicted in red.

\begin{figure}[htbp] 
   \centering
   \includegraphics[width=1.5in]{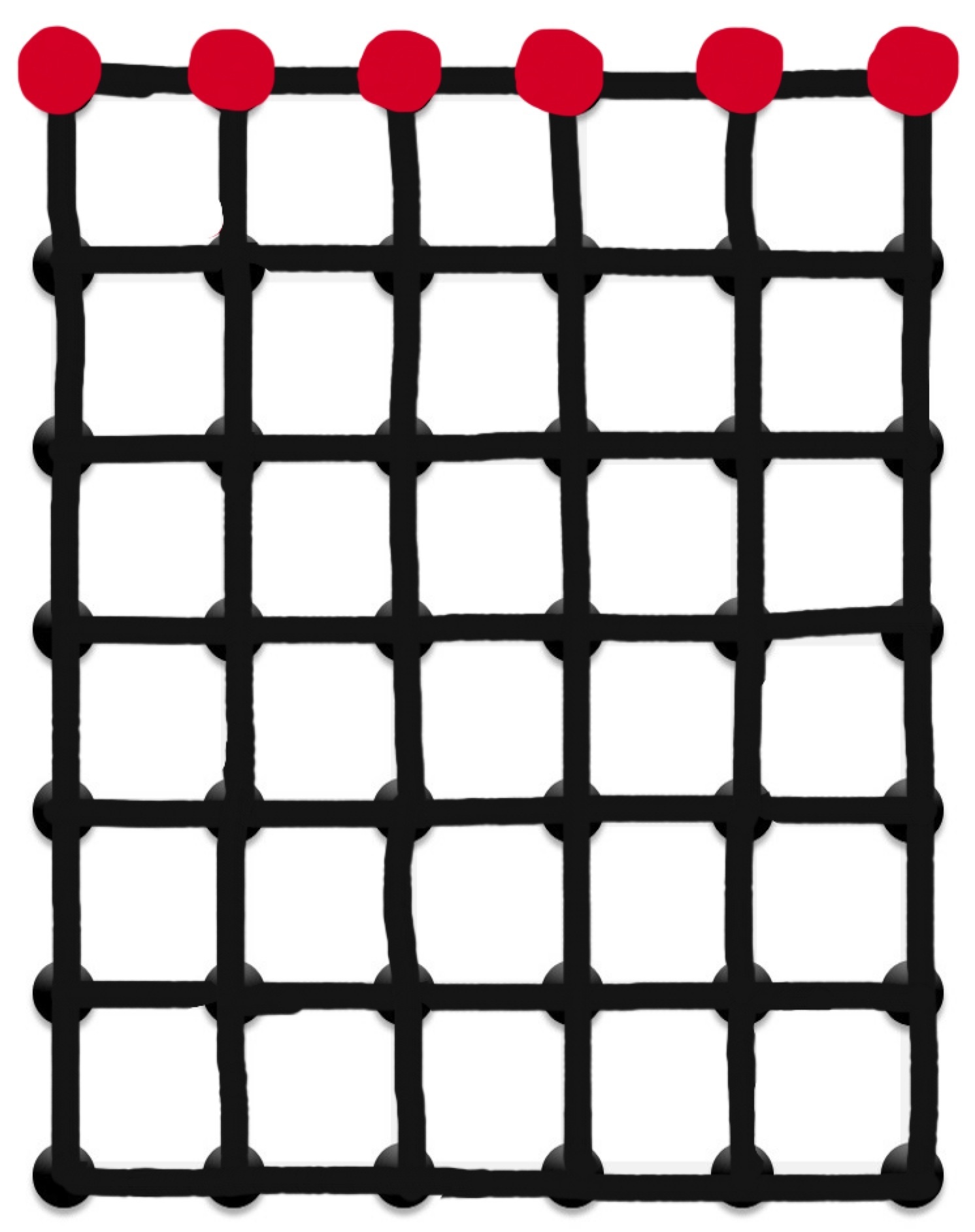}
   \hspace{1in} 
      \includegraphics[width=1.5in]{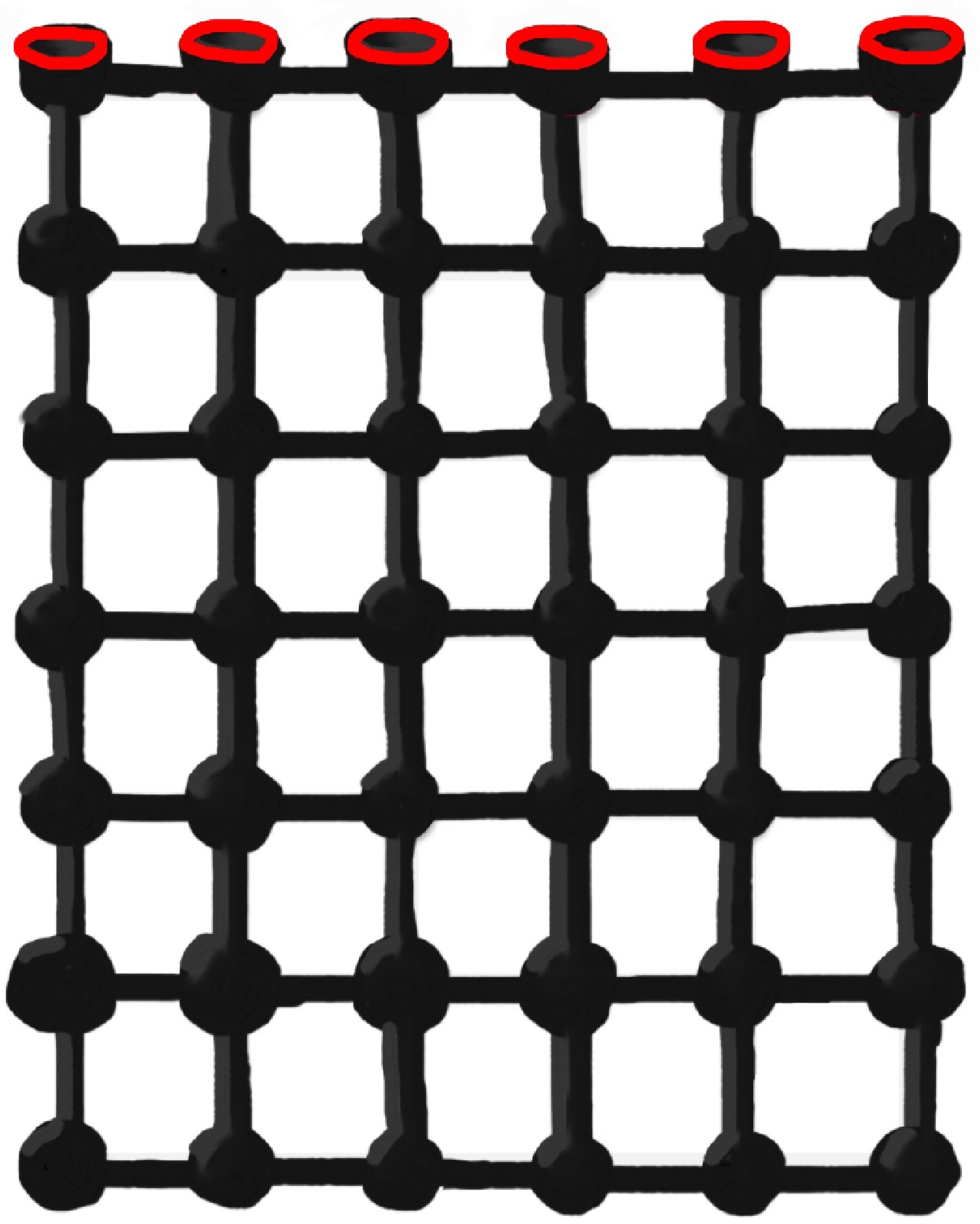} 
 \caption{\textcolor{red}{$A_\vare$}$\,\subset Y_\vare$
 and \textcolor{red}{$\partial M$} $\subset M$ as in the proof
 of Proposition~\ref{prop-lattice}. \newline
 {\em Note that these graphics will be improved before publication.}}  
   \label{fig-taxi-lattice}
\end{figure}

Next we construct a smooth
surface $M$ by replacing the lattice points in $A_\vare \subset Y'_\vare$
by small hemispheres of diameter $<<\vare$ and
lattice points in $Y'_\vare \setminus A_\vare$ by small spheres
of diameter $<<\vare$.   We replace the line segments in $Y_\vare$
by arbitrarily thin cylinders of the same length, small enough that we can 
glue them to their corresponding spheres smoothly replacing disjoint
balls in those spheres or hemispheres.
This creates a smooth manifold, $M$, such that $\partial M$ is
a union of the boundaries of the hemispheres such that
\be
d_{GH}\left(Y_\vare, M\right) < \vare/2 \textrm{ and } 
d_{GH}\left(A_\vare, \partial M\right) < \vare/2.
\ee 
See the right side of Figure~\ref{fig-taxi-lattice}, where $M^2$ is depicted
in black and $\partial M^2$ is in red.
This completes the first claim in the proposition.

To complete the rest, we take $M_k$ consisting of tubes joined at
spheres and hemispheres close to $X_k$ as above such that
\be
d_{GH}\left(X_k, M_k\right) < \vare/k 
\textrm{ and } d_{GH}\left(A_k, \partial M_k\right) < \vare/k.
\ee 
Note that in the construction above we could have created $B'_k\subset Y'_k$
corresponding to $B_k$.  We have $\vare/(2k)$ almost distance preserving
maps $\psi'_k: B'_{k+1}\to B'_{k}$.   So now we glue together the $M_k$
to form $M$ as follows.   If $b\in B'_k$ maps to $\psi'_k(b)\in B'_k$ we
connect the sphere or hemisphere corresponding to $b$ in $M_k$
to a sphere or hemisphere corresponding to $\psi_k(b)$ in $M_{k+1}$
by a very short, very thin tube.
\end{proof}

\subsection{Review of Gromov's Compactness Theorem}

In \cite{Gromov-metric}, Gromov proved a compactness theorem
for sequences of compact metric spaces.  We review this theorems and
related propositions here.

\begin{thm}\label{GH-compactness-1}[Gromov]
Given a $D>0$ and a function $N: (0,D] \to \N$, we define the
collection, $\mathcal{M}^{D,N}$ of compact metric spaces, $(X, d_X)$ with
diameter $\le D$ that can be covered by $N(\epsilon)$ balls of
radius $\epsilon>0$:
\be
X \subset \bigcup_{i=1}^{N(\epsilon)} B_{x_i}(\epsilon).
\ee
This collection $\mathcal{M}^{D,N}$ is compact with respect to the
Gromov-Hausdorff distance.
\end{thm}

It is standard to determine whether a metric space lies in such
a compact collection by examining maximal collections of
disjoint balls:

\begin{prop}\label{disjoint-to-cover}
Given a metric space $(X, d_X)$.  Let $N$ be the
maximum number of pairwise disjoint balls of radius $\epsilon/2$ that
can lie in $X$.  Then the minimum number of balls of radius
$\epsilon$ required to cover $X$ is $\le N$.
\end{prop}

\begin{proof}
Let $\{B_{x_i}(\epsilon/2): \, i=1,...,N\}$ be a maximal collection of
pairwise disjoint balls of radius $\epsilon/2$.  Let $x\in X$.
Then $\exists i \in \{1,...,N\}$ such that $B_{x_i}(\epsilon/2)\cap B_x(\epsilon/2)\neq \emptyset$.  Thus $d_X(x, x_i)< \epsilon$ and
\be
X \subset \bigcup_{i=1}^N B_{x_i}(\epsilon).
\ee
\end{proof}

In a Riemannian manifold or metric measure space, the
volumes of balls may thus be applied to determine the function, $N$.

\begin{prop}\label{unif-lower}
If there exists $\Theta>0$ such that 
\be
\vol(B_p(\epsilon))/\vol(M) \ge \Theta
\ee
then the maximum number of disjoint balls of
radius $\epsilon$ is $\le 1/\Theta$.
\end{prop}

\begin{proof}
\be
\vol(M) \ge \sum_{i=1}^N \vol(B_{x_i}(\epsilon)) \ge \sum_{i=1}^N \Theta \vol(M)
=N \Theta \vol(M).
\ee
\end{proof}

Gromov applies his compactness theorem in conjunction with these
propositions to study the compactness of sequences of compact Riemannian
manifolds for which one is able to control the volumes of balls.   We will
apply the same idea to study sequences of metric completions of open
manifolds.

One of the beauties of Gromov's Compactness Theorem, is that he
has proven the converse as well:

\begin{thm}\label{GH-converse} [Gromov]
Suppose $(X_j, d_j)$ are compact metric spaces.  Suppose that
there exists $\epsilon_0>0$ such that $X_j$ contains at least $j$ disjoint
balls of radius $\epsilon_0$. Then no subsequence of the $X_j$
has a Gromov-Hausdorff limit.   
\end{thm}

In particular, if $(X_j, d_{X_j}) \GHto (X, d_X)$ then they have
a uniform upper bound on diameter.   Nor can they have
many splines, as in the following example:

\begin{ex} \label{ex-many-splines}
Let 
\be
M_j = \{(\theta,r): \,\,\, \theta\,\in \,S^1,\,\, r \,\in \,(\,1, 3+\cos(j\theta) \,) \,\}
\ee
with metric $g_j = dr^2 + r^2 d\theta^2$.  Then $\vol(M_j) \le \pi 4^2$
and $\diam(M_j) \le 3+\pi+3$ with $0$ sectional curvature.

Observe that in $M_j$, the balls of radius $1$ about 
$(2\pi k /j, 3)$ are disjoint because paths between these points
in $M_j$ must reach within $r\le2$ between the
splines and so have length $\ge 2(3-2)$.
Thus there are $j$ disjoint balls of radius $1$ in $M_j$
and no subsequence of the metric completions of
$M_j$ converge in the Gromov-Hausdorff
sense.
\end{ex}

\begin{ex}\label{ex-many-pages}
Let
\be
X_j = \big( [0,1]\times[0,1]\big) \,\disjointunion \, \big( [0,1]\times [0,1/2] \big)
\,\disjointunion \, \cdots \,\disjointunion \, \big( [0,1]\times [0,1/2^j] \big)
\ee
be a disjoint union of spaces with taxicab metrics glued with a gluing
map $\psi(0,y)=(0,y)$.  Then $X_j$ has no Gromov-Hausdorff
converging subsequence because it has $j$ disjoint balls of radius $1$ about points
$(1,0)$.   If we take surfaces $M_j$ as constructed in Proposition~\ref{prop-lattice}, such that
\be
d_{GH}(M_j, X_j)\to 0,
\ee
they also have no Gromov-Hausdorff converging subsequence.   
\end{ex}

Defining an appropriate compact metric space and applying   
Theorem~\ref{Blaschke}, in a later paper, \cite{Gromov-81a}(page 65), Gromov proved the following useful theorem.

\begin{thm} [Gromov] \label{common-Z}
If one has a sequence of compact metric spaces, $(X_j, d_{X_j})$,
such that $(X_j, d_{X_j}) \GHto (X_\infty, d_{X_\infty})$, then
there exists a common compact metric space $Z$ and
isometric embeddings $\varphi_j :(X_j,d_{X_j}) \to (Z, d_Z)$
such that $d_H(\varphi_j(X_j), \varphi_\infty(X_\infty))\to 0$.
\end{thm}

\begin{thm} [Blaschke] \label{Blaschke}
If $Z$ is a compact metric space then every sequence of closed
subsets of $Z$ has a subsequence that converges in Hausdorff sense 
to a closed subset. 
\end{thm} 

Theorem ~\ref{common-Z} implies the Gromov-Hausdorff Arzela-Ascoli Theorem:

\begin{thm}[Gromov] \label{GH-Arz-Asc}
If $X_j \GHto X$ and $Y_j \GHto Y$ and $f_j: X_j \to Y_j$
are equicontinuous,
\be
\forall \epsilon > 0\, \exists \delta_{\epsilon}>0 \textrm{ such that }
d_{X_j}(p,q)< \delta_\epsilon \implies d_{Y_j}(f_j(p), f_j(q)) < \epsilon,
\ee
then there is a subsequence with a continuous limit function
\be
f: X\to Y.
\ee
If the $f_j$ are isometric embeddings, then so is $f$.
\end{thm}

In particular, if the $X_j$ are geodesic spaces, then
so is the limit space \cite{Gromov-metric}.  

\subsection{Gromov's Ricci Compactness Theorem}

In this section we review Gromov's Ricci Compactness Theorem
which is based on the Bishop-Gromov Volume Comparison Theorem 
\cite{Gromov-metric}:

\begin{thm} [Bishop-Gromov]
If $M$ is an $m$ dimensional Riemannian manifold with boundary 
that has nonnegative Ricci curvature
and $B_p(R)\subset M^m$
does not reach the boundary, then for all $r\in (0,R)$ we have
\be
\frac{\vol(B_p(r))}{\vol(B_p(R))} \ge \left( \frac{r}{R} \right)^m
\ee
\end{thm}

Gromov's Ricci Compactness Theorem was originally
stated for compact manifolds without boundary:

\begin{thm} [Gromov]\label{BGcomparison}
Let $m\in \N$, $D>0$ 
and let $\mathcal{M}^{m,D}$ be the class of compact
$m$ dimensional
Riemannian manifolds, $M$, with nonnegative Ricci curvature
and $\diam(M) \le D$.   Here the manifolds do not have boundary.
Then $\mathcal{M}^{m,D}$ is precompact with respect
to the Gromov-Hausdorff distance.
\end{thm}

In fact, Gromov's Compactness Theorem has a commonly
used version applied to balls which we state as follows:

\begin{thm} [Gromov] \label{balls}
Let $m\in \N$, $D>0$ 
and let $\mathcal{M}^{m}$ be the class of compact
$m$ dimensional
Riemannian manifolds, $M$, with nonnegative Ricci curvature.   
Suppose $M_j \in \mathcal{M}^{m}$
and suppose $B_{p_j}(D)$ do not reach the boundary,
then there exists a subsequence such that $(B_{p_j}(D/3), d_{M_j})$
converges in the Gromov-Hausdorff distance.
\end{thm}

For completeness of exposition we show how Gromov's original
proof implies Theorem~\ref{balls}.

\begin{proof}
Let $q\in B_{p_j}(D/3)$.   Then 
\be
B_{p_j}(D/3) \subset B_q(2D/3)\subset B_{p_j}(D)
\ee
does not reach the boundary of $M_j$, so we may apply
the Bishop-Gromov Volume Comparison Theorem to see
that:
\begin{eqnarray}
\frac{\vol(B_q(r))}{\vol(B_{p_j}(D/3))} 
&\ge& \frac{r^m}{(2D/3)^m}\frac{ \vol(B_q(2D/3))}{\vol(B_{p_j}(D/3))} \\
&\ge& \frac{r^m}{(2D/3)^m}\frac{ \vol(B_{p_j}(D/3))}{\vol(B_{p_j}(D/3))} 
=\frac{(3r)^m}{(2D)^m}.
\end{eqnarray}
So now we may apply Proposition~\ref{unif-lower} to
complete the proof.
\end{proof}

\subsection{Volume Convergence Theorems}\label{subsect-vol-conv}

In \cite{Colding-volume}, Colding proved the following 
volume convergence theorem:

\begin{thm}[Colding] \label{thm-colding}
Let $M_j^m$ be complete Riemannian manifolds
with nonnegative Ricci curvature and $p_j\in M_j$ such that 
\be
B_{p_j}(1) \GHto B_0(1) \subset \E^m
\ee
where $\E^m$ is Euclidean space of dimension $m$, then
\be
\lim_{j\to\infty} \vol(B_{p_j}(1))= \vol(B_0(1)).
\ee 
\end{thm}

\begin{rmrk}\label{rmrk-colding-1}
The proof of this theorem does not in fact require global nonnegative
Ricci curvature on a complete manifold.   In fact $M_j^m$ could
be an open manifold as long as $B_{p_j}(2)\subset \bar{M}_j^m$
does not hit the boundary.   In fact one may not even need a radius
of $2$.
\end{rmrk}

Colding applied this theorem to prove a number of theorems including
one in which the Gromov-Hausdorff limit is an arbitrary compact 
Riemannian manifold of the same dimension (also \cite{Colding-volume}):

\begin{thm}[Colding] \label{thm-colding-2}
Let $M_j^m$ and $M_\infty^m$ be compact Riemannian manifolds
with nonnegative Ricci curvature for $j=1,2,3,...$ such that
\be
M_j^m \GHto M_\infty^m.
\ee
Then for all $r>0$ and for all $p_j\in M_j$ such that $p_j \to p_\infty$ we have
\be
\lim_{j\to\infty} \vol(B_{p_j}(r))= \vol(B_{p_\infty}(r)).
\ee 
\end{thm}

\begin{rmrk}\label{rmrk-colding-2}
Again Colding's proof does not really require $M_j$ to be complete.
These $M_j$ could be open Riemannian manifolds as long
as $B_{p_j}(r) \subset \bar{M}_j$ does not hit the boundary.
Here we do not need to worry about twice the radius because the
proof involves estimating countable collections of small balls 
$B_{q_{j,i}}(\epsilon_{j,i})$ in
$B_{p_j}(r)$ and applying Theorem~\ref{thm-colding} to those
small balls and one can always ensure the 
$B_{q_{j,i}}(2\epsilon_{j,i})$ avoid the boundary
as in Remark~\ref{rmrk-colding-1}
\end{rmrk}

Cheeger-Colding then conducted a study of the properties
of Gromov-Hausdorff limits of manifolds of nonnegative
Ricci curvature in \cite{ChCo-PartI}.   They improve
upon Theorem~\ref{thm-colding-2}, allowing $M_\infty$ to be an 
arbitrary limit space as long as the sequence is noncollapsing:

\begin{thm}[Cheeger-Colding] \label{thm-chco}
Let $V_0>0$ and let $M_j^m$ be compact Riemannian manifolds
with nonnegative Ricci curvature for $j=1,2,3,...$, such that
\be
M_j^m \GHto M_\infty^m \textrm{ and } \vol(M_j^m) \ge V_0.
\ee
Then for all $r>0$ and for all $p_j\in M_j$ such that $p_j \to p_\infty\in M_\infty$ we have
\be
\lim_{j\to\infty} \vol(B_{p_j}(r))= \mathcal{H}^m(B_{p_\infty}(r))
\ee 
where $\mathcal{H}^m$ is the Hausdorff measure of dimension $m$.
\end{thm}

\begin{rmrk} \label{rmrk-chco}
Again this theorem is proven locally, so as in Remark~\ref{rmrk-colding-2}
this theorem holds when $M_j^m$ are open Riemannian manifolds
as long as $B_{p_j}(r)\subset \bar{M}_j^m$ do not touch the boundary.
\end{rmrk}

Of course, Cheeger and Colding study more than just manifolds
with nonnegative Ricci curvature and more than just noncollapsing
sequences in their work, but these theorems are the only ones needed 
in this paper.   See also work of the second author with Wei for
an adaption of their volume convergence theorem which deals with
Hausdorff measures defined using restricted vs intrinsic distances
\cite{SorWei1}.

\section{Properties of Inner Regions} \label{sect-inner}

Given an open Riemannian manifold, $M$, we have
defined the $\delta$ inner region, $M^\delta$ in 
Definition~\ref{defn-1}.  Note that these spaces are open
Riemannian manifolds, however we will study them using
the restricted distance, $d_M$, rather than the intrinsic
length metric, $d_{M^\delta}$, defined in (\ref{intrinsic-distance}).
There are natural isometric embeddings of
$(M^\delta, d_M)$ and its metric completion
$(\bar{M}^\delta, d_M)$ into $(M, d_M)$.  Thus the metric
completion is, in fact, compact when $M$ is precompact.
This occurs, for example, when $M$ has finite diameter.

\begin{ex}\label{ex-spline-1}
In Figure~\ref{graph-spline-1}, we depict a single flat manifold,
$M^2$, which is a flat disk with a spline attached.  For a sequence
of $\delta_1<\delta_2<\delta_3< \delta_4$, the grey inner regions
depict $M^{\delta_i}$.  For $\delta$ sufficiently large $M^\delta$
is an empty set. 
\end{ex}

\begin{figure}
\includegraphics[width=100mm]{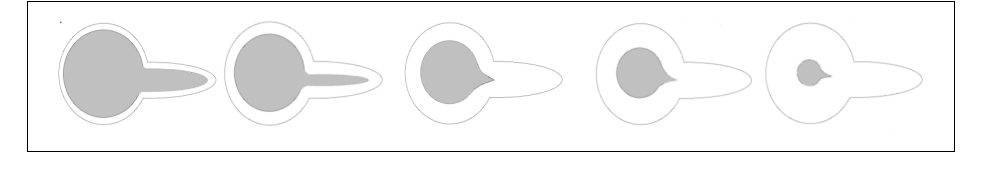}
\label{graph-spline-1}
\caption{Example~\ref{ex-spline-1}: Single $M$ varying $\delta$}
\end{figure}

\begin{lem} \label{lem-exhaust}
For any sequence $\delta_i \to 0$, we have
\be 
M=\bigcup_{i=1}^\infty M^{\delta_i}.
\ee
In fact,
\be
M = \bigcup_{\delta>0} M^\delta.
\ee
\end{lem}

\begin{proof}
Let $x \in M$, since $M$ is open $\varepsilon=d_M(x, \partial M)> 0$.
Then $x \in M^{\varepsilon/2}$. 
\end{proof}

\begin{lem} \label{ball-in-M-delta}
Let $\delta>\delta'>0$. If $y\in M^\delta$ 
then for any $\varepsilon < \delta-\delta'$ we have
\be
B(y,\vare)=\{x\in M: \,d_M(x, y)<\vare\} \subset
M^{\delta'}.
\ee
\end{lem}

\begin{proof}
Let $x\in B(y,\vare)$, so $d_M(x, y)<\delta-\delta'$.
Since $y\in M^\delta$, for all $z\in \partial M$
$d_M(y,z)>\delta$.  By the triangle inequality,
\be
d_M(x,z)\ge d_M(y,z)-d_M(x,y)> \delta -(\delta-\delta')=\delta'.
\ee
\end{proof}

Inner regions, $M^\delta$, with restricted metrics, $d_M$,
are not necessarily length spaces.  

\begin{ex}  
In the
flat open manifold
\be
M=\{(x,y): x^2+y^2 \in (1,25)\} \subset \E^2
\ee
the distance between $(3,1)$ and $(-3,1)$ is
\be
d_M\big( (3,1), (-3, 1) \big) = 6
\ee
because they are joined by curves of length
arbitrarily close to $6$.   However for $\delta=1$
we have
\be
M^\delta=\{(x,y): x^2+y^2 \in (4,16)\} \subset \E^2.
\ee
The length of any curve in $M^\delta$ 
between $(3,1)$ and $(-3,1)$ must go around
$(0,2)$ and thus has length at least $2\sqrt{9+1}>6$.
\end{ex}

In fact inner regions of path connected manifolds
need not be connected:

\begin{ex} 
Let our manifold be the connected union of balls
in the Euclidean plane:
\be
M= B_{(4,0)}(5) \cup B_{(-4,0)}(5) \subset \E^2
\ee
Then
\be
\partial M= A_+ \cup A_-
\ee
where
\begin{eqnarray}
A_+ &=& \partial B_{(4,0)}(5) \cap\{(x,y): x\ge 0\}\\
A_- &=& \partial B_{(-4,0)}(5) \cap\{(x,y): -x\ge 0\}
\end{eqnarray}
Note that
\be
(0, 3), (0, -3) \in \partial M.
\ee
Thus for $\delta >3$, 
\be
M^\delta \cap \{(0,y): y \in \R\} = \emptyset.
\ee
However for $\delta<5$, we have
\be
(4,0), (-4,0) \in M^\delta.
\ee
Thus $M^\delta$ is not connected for $\delta\in (3,5)$.
\end{ex}

\section{Manifolds with Gromov-Haudorff limits have Converging Inner Regions}
\label{sect-manifolds}

In this section we will prove:

\begin{thm}\label{GH-lim-gives-delta-lim}
Suppose $M_j$ are precompact open metric spaces, $(X, d_X)$ is a compact metric 
space and
$(\bar{M}_j, d_{M_j}) \GHto (X, d_X)$ then for each $\delta>0$,
there exist a subsequence $\bar{M}_{j_k}^\delta$, a compact
metric space $Y^{\delta}(j_{k}) \subset X$ 
such that
\be \label{GH-lim-gives-delta-lim-1}
\Big(\bar{M}_{j_k}^\delta, d_{M_{j_k}}\Big) \GHto 
\Big(Y^\delta(j_{k}), d_\delta\Big),
\ee 
for any $\delta=\delta_i$ in the sequence
and 
\be
Y^{\delta_1}(j_{k}) \subset Y^{\delta_2}(j_{k})
\ee
if (\ref{GH-lim-gives-delta-lim-1})  hold for $\delta_{1},\delta_{2}$, 
$0 < \delta_2 <\delta_1$.

Let 
\be \label{defn-U-here}
U_{\{\delta_{i}\},\{j_k\}} = \bigcup_{\delta_i} Y^{\delta_i}(j_{k}).
\ee
Then $U_{\{\delta_{i}\},\{j_k\}}$ is an open subset of $X$.

If we have two sequences $\{\delta_i\}, \{\beta_i\}$
such that (\ref{GH-lim-gives-delta-lim-1}) for all 
$\delta \in \{\delta_i\}\cup\{\beta_i\}$
then
\be\label{same-U-diff subseq}
U_{\{\delta_{i}\},\{j_k\}}=U_{\{\beta_{i}\},\{j_k\}}.
\ee
\end{thm}

Note that $M_j^\delta$ can be an empty space. See Example~\ref{ex-to-empty}.
Consider the Gromov-Hausdorff limit of
an empty metric space, to be an empty metric space.

\begin{rmrk}
In Example~\ref{ex-why-subseq-1} we see that a subsequence
$j_k$ may be necessary to obtain GH convergence of the $\delta$ inner regions and that $U_{\{\delta_{i}\},\{j_k\}}$ depends on the choice of the
subsequence.  In Example~\ref{ex-why-subseq-2} we see that even
the closure of $U_{\{\delta_{i}\},\{j_k\}}$ may depend on the choice
of subsequence $j_k$.  In Example~\ref{ex-why-subseq-3} we see that 
$U_{\{\delta_{i}\},\{j_k\}}$ may be disjoint and not isometric.
\end{rmrk}

\subsection{Hausdorff Convergence of $\delta$ Inner Regions}

We begin with a very basic theorem:

\begin{thm}\label{H-lim-gives-delta-lim}
Let $(Z, d_{Z})$ be a compact metric space. 
Suppose $M_j \subset Z$ are open metric spaces with the induced metric 
and $X \subset Z$ is closed such that $\bar{M}_j \Hto X$ then for each $\delta>0$,
there exist a subsequence $\bar{M}_{j_k}^\delta$ and a compact
set $W^{\delta}(j_k) \subset X$ 
such that
\be \label{H-lim-gives-delta-lim-1}
\bar{M}_{j_k}^\delta \Hto W^{\delta}(j_k),
\ee 
and if (\ref{H-lim-gives-delta-lim-1}) holds for $\delta_{1},\delta_{2}$, 
$0 < \delta_2 <\delta_1$, then
\be
W^{\delta_1}(j_k) \subset W^{\delta_2}(j_k)
\ee

Given a sequence of positive numbers $\delta_{i}\to 0$ there exists a subsequence
$\{j_k\} \subset \N$ such that (\ref{H-lim-gives-delta-lim-1}) holds for all $\delta=\delta_i$. 
Let 
\be
U'_{\{\delta_{i}\},\{j_k\}} = \bigcup_{\delta_i} W^{\delta_i}(j_k)
\ee
$U^{'}_{\{\delta_{i}\},\{j_k\}}$ is an open subset of $X$.
If we have two sequences $\{\delta_i\}, \{\beta_i\}$
such that (\ref{H-lim-gives-delta-lim-1}) for all 
$\delta \in \{\delta_i\}\cup\{\beta_i\}$
then
\be
U'_{\{\delta_{i}\},\{j_k\}}=U_{\{\beta_{i}\},\{j_k\}}.
\ee
\end{thm}

There are occasions where $M_j^\delta$ can be an empty space.
We consider the Hausdorff limit of an empty metric space,
to be an empty metric space.

Before we prove this theorem we provide an example
demonstrating that even if 
\be
\bar{M}_{j_k}^{\delta_1} \Hto W^{\delta_1}(j_k) \,\,\textrm{ and }\,\, 
\bar{M}_{j_k}^{\delta_2} \Hto W^{\delta_2}(j_k),
\ee
for some $\delta_1>\delta_2 >0$, we cannot assure that for $\delta \in (\delta_2,\delta_1)$, $\bar{M}_{j_k}^\delta$ converges:

\begin{ex}\label{ex-no-diag}
Fix $\varepsilon < 1/3$. In 2 dimensional Euclidean space, $\E^2$, consider the sequence $M_{j}$ where
$M_{2j}$ is a ball of radius $1$ with a spline of width $4\varepsilon$ attached to it as it is depicted in Figure~\ref{graph-spline-1}, $M_{2j+1}$ 
is a ball of radius $1$ with a spline whose width decreases from $6\varepsilon$ to $4\varepsilon$ as $j \to \infty$. Then 
$\bar{M}^{\varepsilon}_{j}$ converges to ball of radius $1-\varepsilon$ with a spline of width $2\varepsilon$, $\bar{M}^{3\varepsilon}_{j}$ converges to a ball of radius $1-3\varepsilon$ with no spline attached. But $\bar{M}^{2\varepsilon}_{2j}$ converges to ball of radius $1-2\varepsilon$ while $\bar{M}^{2\varepsilon}_{2j+1}$ converges to ball of radius $1-2\varepsilon$ with a line segment attached to it. Thus $M^{2\varepsilon}_{j}$ does not converge
in the Hausdorff sense.
\end{ex}

In the proof of Theorem~\ref{H-lim-gives-delta-lim} we will apply the
following fact:

\begin{rmrk}
Recall that if $\{ A_{j} \}$ is a sequence of closed subsets
of a metric space $A$ such that $A_{j} \Hto A_{\infty}$, then 
\be
A_{\infty}=\left\{a\in A :\,\,\forall j \in \N \,\,\exists a_{j} \in A_{j}, \lim_{j \to \infty}a_{j}=a\right\}.
\ee
Any subsequence $\{ A_{j_k} \}$ of $\{ A_{j} \}$ also converges
in Hausdorff sense to $A_\infty$. Then  
\be
A_{\infty}=\left\{a\in A :\,\, \forall k \in \N \,\,\exists a_{j_k} \in A_{j_k}, \lim_{k \to \infty}a_{j_k}=a\right\}.
\ee
\end{rmrk}

We now prove Theorem~\ref{H-lim-gives-delta-lim}: 

\begin{proof}
Apply Theorem~\ref{Blaschke} to the sequence $\{\bar{M}_{j}^\delta\}_{j=1}^{\infty}$ to get a subsequence $\{\bar{M}_{j_k}^\delta\}_{k=1}^{\infty}$
and a compact set $W^{\delta}(j_k)$ such that (\ref{H-lim-gives-delta-lim-1}) is satisfied.
Since $\bar{M}_{j_k}^\delta \subset \bar{M}_{j_k}$ then
$W^{\delta}(j_k) \subset  X$. Similarly, $W^{\delta_1}(j_k) \subset W^{\delta_2}(j_k)$ when 
(\ref{H-lim-gives-delta-lim-1}) holds for $0< \delta_{2} < \delta_{1}$.

Given $\delta_{i}\to 0$. Start with $\delta_{1}$. By Theorem~\ref{Blaschke}
there exists a subsequence $\{j_{k}(\delta_{1})\}_{k=1}^{\infty}$ of $\{j\}_{j=1}^{\infty}$ and a compact
set $W^{\delta_1}(j_{k}(\delta_{1}))$ such that
$\bar{M}_  {{j_{k}(\delta_{1})}}
^{\delta_1} \Hto W^{\delta_1}(j_{k}(\delta_{1}))$.
For $n>1$, there exists a subsequence $\{j_{k}(\delta_{n})\}_{k=1}^{\infty}$ of
$\{j_{k}(\delta_{n-1})\}_{k=1}^{\infty}$
and a compact set $W^{\delta_n}(j_{k}(\delta_{n}))$ such that 
$\bar{M}_{j_{k}(\delta_{n})}^{\delta_n} \Hto W^{\delta_n}(j_{k}(\delta_{n}))$.
Define $j_{k}=j_{k}(\delta_{k})$.  Then
$\{ j_{k} \}_{k=n}^{\infty}$ is a subsequence of $\{j_{k}(\delta_{n})\}_{k=1}^{\infty}$
thus (\ref{H-lim-gives-delta-lim-1}) holds for all $n$.

Let $y$ be an element of $U^{'}_{\{\delta_{i}\},\{j_k\}}$.
Then there exist $N \in \N$ such that $y \in W^{\delta_{i}}(j_{k})$
for $i \geq N$. Suppose that $x \in X$ and $d_{Z}(x,y)< \delta_{N}/6$.
Since $y \in W^{\delta_{i}}(j_{k})$ choose $y_{j_{k}} \in \bar{M}^{\delta_N}_{j_{k}}$
such that $y=\lim_{j \to \infty} y_{j_{k}}$ and $d_{Z}(y,y_{j_{k}}) <  \delta_{N}/6$.
Analogously, take $x_{j} \in \bar{M}_{j}$ such that $x=\lim_{j \to \infty} x_{j}$
and $d_{Z}(x,x_{j}) <  \delta_{N}/6$. Then
\be
d_{Z}(x_{j_k},y_{j_{k}}) < d_{Z}(x_{j_k},x) +d_{Z}(x,y)+d_{Z}(y,y_{j_{k}})< \delta_{N}/2 .
\ee
This implies that $d_{Z}(x_{j_k}, \partial(M_{j_{k}})) > \delta_{N}/2$.
Then $x \in W^{\delta_i}(j_{k}) \subset U^{'}_{\{\delta_{i}\},\{j_k\}}$ for some $i > N$.

If there is another sequence $\beta_{i} \to 0$ such that (\ref{H-lim-gives-delta-lim-1}) holds for all $\delta=\beta_i$, for each $i$ find $l$
such that $\delta_{l} < \beta_i$ then
$W^{\beta_i} (j_{k})\subset W^{\delta_l}(j_{k})$. 
This proves that $U^{'}_{\{\beta_{i}\},\{j_k\}} \subset U^{'}_{\{\delta_{i}\},\{j_k\}}$. The same reasoning works to prove $$U^{'}_{\{\delta_{i}\},\{j_k\}} \subset U^{'}_{\{\beta_{i}\},\{j_k\}}$$
\end{proof}

\begin{defn}\label{defn-H-U}
With the hypothesis of Theorem~\ref{H-lim-gives-delta-lim}
define
\be
U^{'}_{\{j_k\}}= \bigcup W^{\delta}(j_{k})
\ee
where the union is taken over all $\delta$ for which
$\bar{M}^{\delta}_{j_k}$ is a sequence that converges in  
Hausdorff sense to a metric space, $W^{\delta}(j_{k})$, and
\be
U^{'}= \bigcup_{\delta>0} W^{\delta}
\ee
where $W^{\delta}$ is the Hausdorff limit space
of some convergent subsequence of $\bar{M}^{\delta}_{j}$.
\end{defn}

\subsection{Finding Limits of Inner Regions in the Gromov-Haudorff limits}

In this subsection we prove Theorem~\ref{GH-lim-gives-delta-lim}:

\begin{proof}
By Theorem~\ref{common-Z} there exists a common metric space $Z$ and isommetric embeddings $\varphi_{j}:(\bar{M}_j, d_{M_j}) \to (Z,d_{Z})$, $\varphi:(X, d_{X}) \to (Z,d_{Z})$ such that $d_{H}^{Z}(\varphi_{j}(\bar{M}_j),\varphi(X)) \to 0$. 
Now we can apply Theorem~\ref{H-lim-gives-delta-lim}.
For each $\delta>0$,
there exist a subsequence $\varphi_{j_k}(\bar{M}_{j_k}^\delta)$ and a compact
set $W^{\delta}(j_k) \subset \varphi(X)$ such that $\varphi_{j_k}(\bar{M}_{j_k}^\delta) \Hto W^{\delta}(j_k)$. Let $Y^{\delta}(j_k)=\varphi^{-1}(W^{\delta}(j_k))$.
Clearly, (\ref{GH-lim-gives-delta-lim-1}) holds and $Y^{\delta_1}(j_k) \subset Y^{\delta_2}(j_k)$ when (\ref{GH-lim-gives-delta-lim-1}) hold for $0 < \delta_2 <\delta_1$.
Given a sequence of positive numbers $\delta_{i}\to 0$ there exists a subsequence
$\{ j_k \} \subset \N$ 
such that 
$\varphi_{j_k}$
$(\bar{M}^{\delta_{i}}_{j_k})$
$\Hto W^{\delta_{i}}(j_k)$ for all $i$, then (\ref{GH-lim-gives-delta-lim-1}) hold for all $i$ and $U_{\{\delta_{i}\},\{j_k\}} =\varphi^{-1}(U^{'}_{\{\delta_{i}\},\{j_k\}})$ is an open subset of $X$ that does not depend on the 
sequence $\delta_{i}$.
\end{proof}

\subsection{Unions of the Limits of Inner Regions in the Gromov-Haudorff limits}

The following notion of a limit's interior union has some interesting
properties:

\begin{defn}\label{defn-GH-U}
We define a "limit's inner union" of a sequence of open Riemannian
manifolds satisfying the hypothesis of Theorem~\ref{GH-lim-gives-delta-lim}
to be
\be \label{defn-U-here-2}
U=U_{\{j_k\}} = \bigcup_{\delta\in D_{\{j_k\}} }Y^{\delta}(j_{k})
\ee
where the union is taken over all $\delta>0$ such that
$M_{j_k}^\delta$ have Gromov-Hausdorff limits as in
Theorem~\ref{GH-lim-gives-delta-lim}.   Observe that for any
sequence $\delta_i \in D_{\{j_k\}}$, by (\ref{same-U-diff subseq})
we have
\be
U=U_{\{j_k\}} = U_{\{\delta_{i}\},\{j_k\}}.
\ee
\end{defn}

In Theorem~\ref{thm-uniquenessGluedLimit}, we will prove that the limit's
inner union, $U$, defined in Definition~\ref{defn-GH-U} is a 
special case of the glued
limits we will construct in Theorem~\ref{def-glue}.   Since it is
easy to understand the properties of these $U$, we present a
few examples of them here so that we may refer to them later
as examples of glued limit spaces.

\begin{ex}\label{ex-to-empty}
Let $M_j$ be a Euclidean disk of radius $1/j$.  Then 
$\bar{M}_j \GHto X$ where $X$ is a single point. For
any $\delta>0$, taking $j> 1/\delta$, we see that
$M_j^\delta$ are empty spaces. Thus $U$ is the empty set.  
\end{ex}

In the following example we see that $U_{\{j_k\}}$
depends on the subsequence $\{j_k\}$ and in Example~\ref{ex-one-spline},
we see that $X$ is not necessarily contained in the closure of $U$.

\begin{ex}\label{ex-why-subseq-1}
Let $M_{2j}$ be the standard Euclidean disk of radius 1 and let
$M_{2j+1}$ be the standard Euclidean disk with the center point
removed.   Then $\bar{M}_j$ is a closed Euclidean disk as is the
limit space $X$.   
Given $\delta\in (0,1)$, $M^\delta_{2j}$ is the standard Euclidean
disk of radius $1-\delta$.  Their metric completions converge to
the closed disk of radius $1-\delta$.  $U_{\{2j\}}$ is the open 
Euclidean disk of radius $1$.
However $M^\delta_{2j+1}$ is
a Euclidean annulus, $Ann_0(\delta, 1-\delta)$, and the metric
completions converge to the closure of this annulus.   
$U_{\{2j+1\}}$ is the open Euclidean disk of radius $1$
with the center point removed. 
In this example $U=U_{\{2j\}}$.
\end{ex}

\begin{ex} \label{ex-one-spline}
In 2 dimensional Euclidean space consider the sequence of a ball with 
a spline attached to it as it is depicted in Figure~\ref{graph-spline-1}. The Gromov-Hausdorff
limit of the sequence is a ball with an interval attached while the closure 
of U is just the closed ball.   
\end{ex}

In Example~\ref{ex-why-subseq-1}, we saw $U_{\{2j\}} \neq U_{\{2j+1\}}$
still their closures are the same. This is not always the case. $U_{\{j_k\}}$
could even be an empty set.

\begin{ex}\label{ex-why-subseq-2}
For $j \in \N$, let $M_{2j}$ be a flat torus so it has no boundary and
$M_{2j+1}$ be flat tori with increasingly dense small holes cut out
but the holes get smaller and smaller so that they still converge to the
flat torus $X$. $U_{\{2j\}}=X$ but for any $\delta>0$, $M_{2j+1}^\delta$
becomes an empty set. So $U_{\{2j+1\}}$ is the empty set.
\end{ex}

\begin{ex}\label{ex-why-subseq-3}
For $j \in \N$, let $M_{2j}$ be a flat torus $S^1\times S^1$,
with increasingly many dense small holes in $W\times S^1$,
where $W=(0,\pi/4) \subset S^1$ and 
let
$M_{2j+1}$ be a flat torus $S^1\times S^1$,
with increasingly many dense small holes in 
$(S^1 \setminus W) \times S^1$.  Then 
\be
U_{\{2j\}}=(S^1\setminus W)\times S^1\,\,\,
\textrm{ and }\,\,\,
U_{\{2j+1\}}=W \times S^1
\ee
with the restricted distance from $S^1\times S^1$
which are disjoint and not isometric to each other.
\end{ex}

\begin{ex}\label{ex-F}
It is possible for a sequence of open Riemannian manifolds, $M_j$,
to have $\delta$ inner regions $M_j^\delta$ which converge in the
Gromov-Hausdorff sense to some $Y^\delta$ for all $\delta>0$
and yet the limit has two distinct inner unions $U_{\{2j\}} \neq U_{\{2j+1\}}$.
This can be seen for example with the following F-shaped regions:
\be
M_j = (0, 1/j)\times(-1,0] \,\,\,  \cup 
\,\,\,(0,1)\times (0,3) \,\,\, \cup
\,\,\,[1,3)\times (0,1) \,\,\, \cup
\,\,\,[1,3)\times (2,3) \,\,\, \setminus A_j
\ee
in the Euclidean plane
where $A_{2j}$ is an increasingly dense increasingly tiny
collection of balls in $(1,3)\times (0,1)$ and 
$A_{2j+1}$ is an increasingly dense increasingly tiny
collection of balls in $(1,3)\times (2,3)$.    Then
\be
M_j \GHto X = \,\,\,(0,1)\times (0,3) \,\,\, \cup
\,\,\,[1,3)\times (0,1) \,\,\, \cup
\,\,\,[1,3)\times (2,3).
\ee
For $\delta>0$ fixed taking $j$ large enough that $1/j<2\delta$
we see that 
\be
U_{2j}^\delta \GHto Y^\delta = (\delta,1-\delta)\times (\delta,3-\delta) \,\,\, \cup
\,\,\,[1-\delta,3-\delta)\times (2+\delta,3-\delta) 
\ee
which is isometric to 
\be
U_{2j+1}^\delta \GHto Y^\delta = (\delta,1-\delta)\times (\delta,3-\delta) \,\,\, \cup
\,\,\,[1-\delta,3-\delta)\times (\delta,1-\delta). 
\ee
Thus $M_j^\delta$ have a GH limit without taking a subsequence.
On the other hand the limit inner regions are not equal:
\be
U_{\{2j\}}= (0,1)\times (0,3) \,\,\, \cup
\,\,\,[1,3)\times (2,3) \subset X 
\ee
and
\be
U_{\{2j+1\}} = (0,1)\times (0,3) \,\,\, \cup
\,\,\,[1,3)\times (0,1)\subset X 
\ee
only isometric.   We will prove in Theorem~\ref{thm-uniquenessGluedLimit}
that when $M_j^\delta$ have GH limits for all $\delta$ then the closures of the limit's inner unions are always isometric.
\end{ex}

\section{Converging Inner Regions of Sequences with Curvature Bounds}
\label{sect-main}

In this section we prove $\delta$ inner regions converge
under certain geometric hypothesis on the manifolds even
when the manifolds themselves have no Gromov-Hausdorff limits.

\subsection{Constant Sectional Curvature}

Here we prove that the inner regions of a sequence of
manifolds in the following class have a subsequence which
converges in the Gromov-Hausdorff sense.

\begin{defn}\label{def-sc-const}
Given $m\in \N$, $H\in \R$, $V>0$, and $l>0$ we
define 
$
\mathcal{M}^{m, V,l}_{H}  
$
to be the class of 
connected open Riemannian manifolds, $M$, of dimension $\le m$,
with constant sectional curvature $Sect_M=H$, $\vol(M)\le V$, and
\be
L_{min}(M)=\inf\{ L_g(C): \,\, C \textrm{ is a closed geodesic in M } \} > l .
\ee
where a closed geodesic is any geodesic which starts and ends at the same point.
\end{defn}

Recall that complete simply connected manifolds with constant sectional curvature $H\le 0$ have no closed geodesics by the Hadamard
Theorem while those with $H>0$ have $L(M)= 2\pi/\sqrt{H}$ (c.f. \cite{do-Carmo}).    Here we are requiring that the closed geodesic lies in an
open manifold $M$ and we do not have completeness.

\begin{thm}\label{sc-const}
Given any $\delta>0$ if  $(M_j, g_j) \subset \mathcal{M}^{m, V,l}_{H}$, then
there is a subsequence $(M^{\delta}_{j_k}, d_{M_{j_k}})$ such that the metric completion
with the restricted metric
converges in the Gromov-Hausdorff sense to a metric space
$(Y^\delta,d)$.  
In particular the extrinsic diameters measured
using the restricted metric are bounded
uniformly
\be
\diam(M^\delta_{j_k}, d_{M_{j_k}})  \le \epsilon_{0} \frac{V}{V^{m}_{H}(\epsilon_0)}
\ee
\be
\diam(Y^\delta,d)\le \epsilon_{0} \frac{V}{V^{m}_{H}(\epsilon_0)}.
\ee
where 
\begin{eqnarray}
\epsilon_{0}&=&min \{\delta,l/2,\pi/\sqrt{H}\}/2 \textrm{ if }H > 0\\
\epsilon_{0}&=&min \{\delta,l/2\}/2 \textrm{ otherwise},
\end{eqnarray}
and $V^{m}_{H}(\epsilon_{0})$
is the volume of a ball of radius $\epsilon_{0}$ in the complete simply connected
space with constant sectional curvature $H$. 
\end{thm}

\begin{rmrk}There are no closed geodesics in the $M_j$
of Examples~\ref{gold-foils} and~\ref{many-splines}, so $L(M_j)=\infty$.
These examples have $H=0$ and $m=2$.  
Since Example~\ref{many-splines} also has a uniform upper bound on
volume,
it demonstrates why we can only obtain Gromov-Hausdorff convergence
of the $M_j^\delta$ instead of the $M_j$ themselves.   The $M_j^\delta$
of Example~\ref{gold-foils} do not have Gromov-Hausdorff converging
subsequences (see Remark~\ref{rmrk-gold-foils}) demonstrating the
necessity of the hypothesis requiring an upper volume bound. 
\end{rmrk}

\begin{proof}
Let $M \in \mathcal{M}^{m, V, l}_{H}$ and
$p\in M^\delta$. Recall that $\epsilon_{0}=min \{\delta,l/2,\pi/\sqrt{H}\}/2$ if $H >0$,
$\epsilon_{0}=min \{\delta,l/2\}/2$ otherwise.
Then for $0<\epsilon<\epsilon_{0}$,
$B_p(\epsilon)$ does not reach the boundary of $M$ and
does not contain any conjugate point to $p$ since
one does not reach a conjugate point before one
would in the comparison space. 

We claim that there are also no cut points to $p$
in $B_p(\epsilon)$.   If there was a cut point, $q$, then
proceeding as in a similar way to Klingenberg~\cite{Klingenberg},
there exists a closed geodesic starting at $p$ 
of length $\le 2d(p,q)<2\epsilon_0$.  By hypothesis, the length of this
closed geodesic is greater than $l$ which is a contradiction.

Thus there is a Riemannian isometric diffeomorphism
\be\label{is-a-const-ball}
\psi: B_p(\epsilon_0) \to B_x(\epsilon_0)\subset M_H^m
\ee
where $M_H^m$ is the simply connected space of constant
sectional curvature, $H$.   In particular
$\vol(B_p(\epsilon))$
is greater or equal than the volume of a ball of the same radius in
a simply connected space form of constant curvature $H$.
By combining Proposition~\ref{unif-lower}
with Proposition~\ref{disjoint-to-cover} and then Gromov's Compactness
Theorem there is a subsequence $(M^{\delta}_{j_k}, d_{M_{j_k}})$ such that the metric completion
with the restricted metric converges in the Gromov-Hausdorff sense to a metric space
$(Y^\delta,d)$.  
Notice that by Proposition~\ref{unif-lower}, the maximum number of disjoint balls of
radius $\epsilon_0/2$ that lie in $M$ is $\le \frac{V}{V^{m}_{H}}(\epsilon_0/2)$. 
Thus, by Proposition~\ref{disjoint-to-cover},
the minimum number of balls of radius $\epsilon_0$ to cover $M$ is $\le \frac{V}{V^{m}_{H}}(\epsilon_0/2)$.
From this follows that 
\be
\diam(M^\delta_{j_k}, {d_{M_{j_k}}})  \le \epsilon_{0} \frac{V}{V^{m}_{H}(\epsilon_0/2)}.
\ee
Since 
\be 
\diam(M^\delta_{j_k}, {d_{M_{j_k}}}) \to \diam(Y^\delta, d),
\ee
we conclude that 
\be
\diam(Y^\delta,d)\le \epsilon_{0} \frac{V}{V^{m}_{H}(\epsilon_0/2)}.
\ee
\end{proof}

\begin{rmrk}\label{rmrk-sc-const}
If the injectivity radius for each $p \in M^{\delta}_j$ is bounded above
by a positive constant then the condition in the lenght of closed geodesics
in Theorem~\ref{sc-const} is satisfied.  
\end{rmrk}

\subsection{Examples with Constant Sectional Curvature}

The volume condition in Theorem~\ref{sc-const} may not
be replaced by a condition on diameter:

\begin{rmrk} \label{rmrk-gold-foils}
Let $(M_j, g_j)$ be the $j^{th}$ covering space of 
$Ann_{(0,0)}(1/j, 1)\subset \E^2$.   

Since every point in $M_j$ is less than a distance
$1$ from the inner boundary, and the inner boundary
has length $j 2\pi (1/j)=2\pi$, we know
\be
\diam(M_j, d_{M_j}) \le 2\pi+2.
\ee
Yet the number of disjoint balls of radius $\delta<1/4$
centered on the cover of $\partial B_{(0,0)}(1/2)$
is greater than $2j$.  So there is no subsequence of
$M^\delta_j$ which converges in the Gromov-Hausdorff sense.

This sequence fails to satisfy the volume condition of
Theorem~\ref{sc-const}:
\be
\vol(M_j)= j (\pi 1^2- \pi/j^2)= \pi (j^2-1)/j.
\ee
It is worth observing that the intrinsic diameters
\be
\diam(M^\delta_j, {M^\delta_j}) \ge j 2 \pi(\delta+1/j)
\ee
also diverges to infinity.   
\end{rmrk}

\begin{rmrk}\label{rmrk-many-splines-delta}
The flat manifolds of Example~\ref{many-splines}
described explicitly in Example~\ref{ex-many-splines} satisfies the hypothesis
of Theorem~\ref{sc-const}.  See Figure~\ref{fig-many-splines-2}.
In fact
for fixed $\delta>0$, once $(2\pi/j) 4 < \delta$, every
point with $r\ge 2$ lies within a distance $\delta$ from the boundary
because the spline is less than $\delta$ wide.  So all the $M^\delta_j$
eventually lie within $r<2$, where the metric is just the standard
Euclidean metric and there is a uniform bound on the number of 
disjoint balls.  So the Gromov-Hausdorff limit also lies within the
Euclidean ball of radius $2$.  On the other hand, every point within
the ball of radius $1+\delta < r < 2-\delta$, lies in $M_j^\delta$, so the Gromov-Hausdorff
limit $Y^\delta$ contains $Ann_{(0,0)}(1+\delta, 2-\delta)$.  In fact
$Y^\delta$ is the metric completion of this annulus with the flat 
Euclidean metric.
\end{rmrk}

\begin{figure}[htbp] 
   \centering
   \includegraphics[width=3.5in]{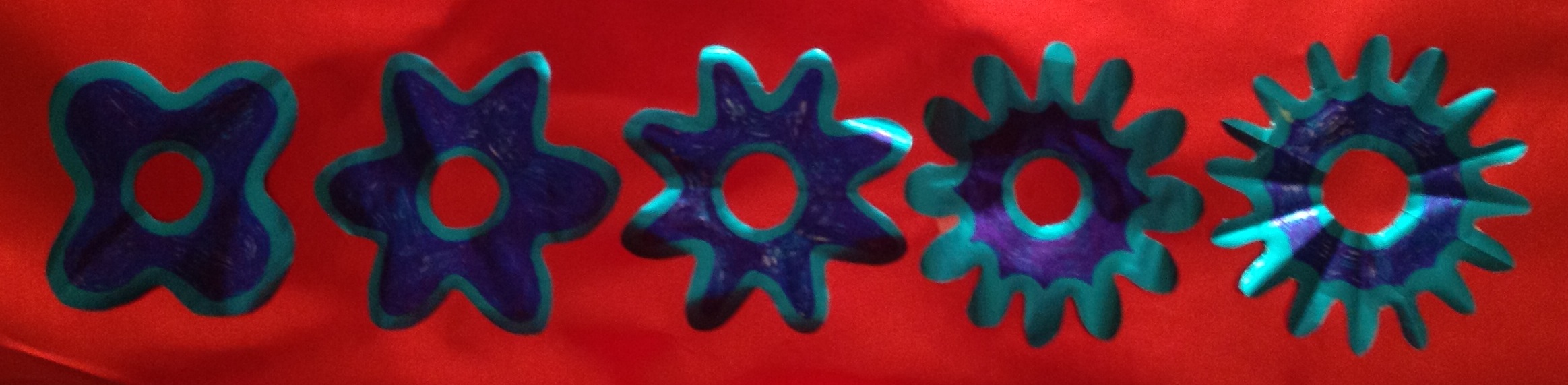} 
 \caption{Models of Example~\ref{many-splines}: $M_4^\delta$, $M_6^\delta$, $M_8^\delta$, $M_{12}^\delta$, $M_{16}^\delta$...}
   \label{fig-many-splines-2}
\end{figure}

\subsection{Manifolds with Nonnegative Ricci Curvature}

Here we prove Theorem~\ref{main-thm} by applying
Gromov's Compactness Theorem (c.f. Theorem~\ref{GH-compactness-1})
combined with the following proposition:


\begin{prop}\label{chain-counting}
If $(M, g_M)$ is a compact Riemannian manifold with boundary
that has nonnegative Ricci curvature, then for any
$\delta>0$, and any $\epsilon \in (0,\delta/2)$, 
the $\delta$ inner region, $M^\delta$, contains a
finite collection of points $\{p_1, p_2,... p_N\}$, such that
\be
M^\delta \subset \bigcup_{i=1}^N B_{p_i}(\epsilon) 
\ee
where
\be
N \le N(\delta, \epsilon, D_\delta, V, \theta)
= \frac{V}{\theta} \,
\left(\frac{2^{2 D_\delta/\epsilon}}{\epsilon}\right)^m
\ee
where $m=\dim(M)$, $\vol(M) \le V$,
\be
\diam(M^\delta,{d_{M^\delta}})\le D_\delta,
\ee
and
\be \label{theta}
\sup \{ \vol(B_q(\delta)): \,q\in M^\delta\} \ge \theta \delta^m .
\ee
\end{prop}

\begin{rmrk}
Note that in this proposition,
we can use the volume of any ball centered in $M^\delta$
to estimate $\theta$ in (\ref{theta}).  This allows us
to study sequences like those in Example~\ref{ex-spline-1}.
One does not
need a Ricci curvature condition if one has a uniform lower
bound on the volumes of all balls centered in $M^\delta$
as can be seen in Proposition~\ref{unif-lower} 
in the review of Gromov-Hausdorff convergence.
\end{rmrk}

\begin{proof}
By Propositions~\ref{disjoint-to-cover} and~\ref{unif-lower} in the
review of Gromov-Hausdorff convergence, we need only to find
a uniform lower bound on the volume of an arbitrary ball
$B_p(\epsilon)$ centered at $p\in M^\delta$.  

Fix $q$ achieving the supremum in (\ref{V-epsilon}).  Then
by the fact that $B_q(\delta)$ does not hit $\partial M$ and
$M$ has nonnegative Ricci curvature, we may apply the Bishop-Gromov
Volume Comparison Theorem to see that,
\be \label{theta-epsilon}
\theta\le\frac{\vol(B_q(\delta))}{\delta^m}
\le \frac{\vol(B_q(\epsilon))}{\epsilon^m},
\ee
because $\delta>\delta/2>\epsilon$.   

Let $C:[0,1]\to M^\delta$
be the shortest curve from $p$ to $q$.  Then 
\be
L=L(C) \le\diam(M^\delta,{d_{M^\delta}})\le D_\delta.
\ee
Let $n> L /\epsilon$ and $x_j=C(t_j)$ where $t_j=j L/n$, so that
\be
x_j \in M^\delta \textrm{ and }d_M(x_{j-1}, x_j)=L/n < \epsilon. 
\ee
In particular $B_{x_j}(2\epsilon)$ lies within the interior of $M$
and has nonnegative Ricci curvature.  Thus by the 
Bishop-Gromov Volume Comparison Theorem,
\begin{eqnarray}
\vol(B_{x_j}(\epsilon)) &\ge& \frac{1}{2^m} \vol(B_{x_j}(2\epsilon) )\\
&\ge& \frac{1}{2^m} \vol(B_{x_{j+1}}(\epsilon) ).
\end{eqnarray}
Applying this repeatedly from $j=1$ to $n$, 
and finally applying (\ref{theta-epsilon}), we have
\begin{eqnarray}
\vol(B_{p}(\epsilon)) &\ge& \frac{1}{2^{mn}} \vol(B_{q}(\epsilon) )\\
&\ge& \frac{1}{2^{mD_\delta/\epsilon}} \vol(B_{q}(\epsilon) )\\
&\ge& \frac{1}{2^{mD_\delta/\epsilon}} \theta \epsilon^m .
\end{eqnarray}
The estimate on $N(\delta, \epsilon, D_\delta, V, \theta)$
then follows immediately from 
Propositions~\ref{disjoint-to-cover} and~\ref{unif-lower}.
\end{proof}

\section{Glued Limit Spaces} \label{sect-glued-limit-spaces}

In this section we define ``glued limit spaces" and ``completed glued limit spaces" providing a
study of their properties without making any curvature assumptions.   
We begin by constructing isometric embeddings
$\varphi_{\delta_{i+1}, \delta_i}: Y^{\delta_{i+1}}\to Y^{\delta_i}$
between the Gromov Hausdorff limits, $Y^{\delta_i}$, of inner
regions, $M_j^{\delta_i}$ [Theorem~\ref{thm-gluing-isoms}].  
We then apply these
isometric embeddings to glue together, $Y^{\delta_i}$, 
and construct a glued limit space, $Y=Y(\{\delta_i\}, \{\varphi_{\delta_{i+1}, \delta_i}\})$ [Theorem~\ref{def-glue}].  

We next study sequence of $M_j$ which converge in the Gromov-Hausdorff sense.   We prove that if the sequence has a completed glued limit space, then it is unique  [Theorem~\ref{thm-uniquenessGluedLimit}].
However the glued limit is not the Gromov-Hausdorff limit [Remark~\ref{rmrk-one-spline-2}], it might even be empty [Remark~\ref{rmrk-disappear-GH}],
and it need not exist [Remark~\ref{GH-limit-subseq-for-glued}]

Finally we construct some important examples of glued limit space for sequences which do not have Gromov-Hausdorff limits.
In Remark~\ref{rmrk-many-splines} we describe how Example~\ref{ex-many-splines} has a bounded 
and precompact glued limit space.   We provide another example
with a bounded glued limit space which is not precompact 
[Example~\ref{ex-many-pages-2}].   We provide an example where
the glued limit space is not a length space [Remark~\ref{glue-space-not-length}]. We close this section with Example~\ref{ex-nonunique} demonstrating that these glued limit spaces and their completions
depend on the isometric embeddings used to define them and need not be unique.

\subsection{Gluing the Inner Regions Together}

Here we prove the existence of isometric embeddings which
we will later apply as glue to connect the inner regions together.

\begin{thm}\label{thm-gluing-isoms}
Suppose that $\delta_i \to 0$ is decreasing, $i=0,1,2,\cdots$,
and one has a sequence of open manifolds $M_j$
such that 
\be
(\bar{M}_j^{\delta_i}, d_{M_j}) \GHto (Y^{\delta_i}, d_{Y^{\delta_i}})
\ee
for all $i$, where possibly some of these sequences
and their limits are eventually empty sets.
Then there exist subsequential limit isometric embeddings:
\be \label{subseqlim}
\varphi_{\delta_{i+1}, \delta_i} : Y^{\delta_i} \to Y^{\delta_{i+1}},
\ee
which are just the identity when $\delta_i=\delta_{i+1}$.
If $\delta \in (0,\delta_0]$ there exists a compact metric space, 
$Y^{\delta_{i}} \subset Y^\delta \subset Y^{\delta_{i+1}}$ with the restricted metric
$d_{Y^\delta}=d_{Y^{\delta_{i+1}}}$, and a converging subsequence
\be
(\bar{M}_{j_k}^\delta, d_{M_{j_k}}) \GHto (Y^\delta, d_{Y^\delta})
\ee 
and when $\delta \in (\delta_{i+1},\delta_{i})$ for any such $Y^\delta$ 
the restriction map,
$\varphi_{\delta, \delta_i} : Y^{\delta_i} \to Y^{\delta}$,
and the inclusion map,
$\varphi_{\delta_{i+1}, \delta} : Y^{\delta} \to Y^{\delta_{i+1}}
$
are isometric embeddings.
\end{thm}

\begin{proof}
By Theorem~\ref{common-Z} 
for each $i$
there exists a compact metric space $Z_i$ and 
isometric embeddings 
\be
\varphi_{j}:\bar{M}^{\delta_{i+1}}_{j} \to Z_i \textrm{ and }
\varphi_{\infty}:Y^{\delta_{i+1}} \to Z_i
\ee
such that
\be
\varphi_{j}(\bar{M}^{\delta_{i+1}}_{j})
\Hto \varphi_{\infty}(Y^{\delta_{i+1}}).
\ee
By Theorem~\ref{Blaschke} we can choose a subsequence $\{j_{k}\}_{k=1}^{\infty}$
such that $\varphi_{j_k}(\bar{M}^{\delta_i}_{j_k})$
converge in Hausdorff sense to a compact subspace $X^{\delta_{i}} \subset \varphi_{\infty}(Y^{\delta_{i+1}})$. By the hypothesis, 
\be
\bar{M}^{\delta_i}_{j_k} \GHto Y^{\delta_{i}}.
\ee
Then by uniqueness up to an isometry of the Gromov-Hausdorff limit space 
there exists an isometric embedding 
\be
\varphi_{\delta_{i+1}, \delta_{i}}:Y^{\delta_i} \to Y^{\delta_{i+1}}
\textrm{ such that }
\varphi_{\delta_{i+1}, \delta_{i}}(Y^{\delta_i})=\varphi_{\infty}^{-1}(X^{\delta_{i}}).
\ee

By Theorem~\ref{GH-converse} there is a uniform upper bound, $D_i>0$,  of the diameters of $(\bar{M}_{j}^{\delta_i},d_{M_ {j}})$ 
and a function $N_i:(0,D_i] \to \N$ such that $N_i(\epsilon)$ is an upper bound for the number of $\epsilon-balls$ needed to cover $\bar{M}_{j}^{\delta_i}$
for all $\epsilon \in (0,D_i]$ and for all $j\in \N$.
If $\delta \in (\delta_{i+1},\delta_i)$, define $N:(0,D_{i+1}/2] \to \N$ by
$N(\epsilon)=N_{i+1}(2\epsilon)$.  Then
\be
\diam(\bar{M}_j^\delta, {d_{M_j}}) \leq 
\diam(\bar{M}_{j}^{\delta_{i+1}}, {d_{M_j}}) \leq D_{i+1}.
\ee 
Apply Theorem~\ref{GH-compactness-1} to
get a subsequence $\{l_{k}\}_{k=1}^{\infty}$ of $\{j_{k}\}_{k=1}^{\infty}$
such that $\varphi_{l_k}(\bar{M}^{\delta_i}_{l_k})$ converge in Hausdorff sense to a
closed subset $X^{\delta} \subset \varphi_{\infty}(Y^{\delta_{i+1}})$. 

We define 
\be
Y^\delta=\varphi_{\infty}^{-1}(X_\delta)\subset Y^{\delta_{i+1}}.
\ee
The choice of 
a subsequence $\{l_{k}\}_{k=1}^{\infty}$ of $\{j_{k}\}_{k=1}^{\infty}$ implies that $X^{\delta_i} \subset Y^{\delta}$.   
So $Y^{\delta_i} \subset Y^{\delta}$.   The rest of the theorem immediately 
follows.
\end{proof}

\begin{rmrk}
The choice of isometric embeddings $\varphi_{\delta_{i+1},\delta_i}$
is not unique.  See Example~\ref{ex-nonunique} where we provide
two distinct isometric embeddings 
$\varphi_{\delta_{i+1},\delta_i}\neq \varphi'_{\delta_{i+1},\delta_i}$.
\end{rmrk}

\subsection{Glued Limit Spaces are Defined}

We now define a glued limit space for a sequence of
Riemannian manifolds satisfying the hypothesis of Theorem~\ref{thm-gluing-isoms}.   We prove that this glued limit space is a metric space unless
it is the empty set.
We prove that it contains isometric images of all Gromov-Haudorff limits of converging
subsequences of inner regions (which may be empty).   
An example of a sequence
of open Riemannian manifolds which has an empty glued limit
space will be given in Remark~\ref{rmrk-disappear-GH}.   Our
definitions of a glued limit space and a completed glued limit space 
are stated along with their construction in the following theorem:

\begin{thm} \label{def-glue}
Given a sequence of open Riemannian manifolds, $M_j$,
with a sequence $\delta_i \to 0$ satisfying
the hypothesis of Theorem~\ref{thm-gluing-isoms},
one can define a ``glued limit space", $Y$, using the subsequential
limit isometric embeddings of (\ref{subseqlim})
as follows:
\be
Y=Y(\{\delta_i\}, \{\varphi_{\delta_{i+1}, \delta_{i}} \} )=
Y^{\delta_0} \,\disjointunion \, \left( \,\disjointunion \,_{i=1}^\infty 
\left(Y^{\delta_{i+1}}\setminus  \varphi_{\delta_{i+1}, \delta_i}\left(Y^{\delta_i}\right)\right)\right)
\ee
with the metric:
\[ 
d_Y(x,y) =\left\{
\begin{array}{ll}
d_{Y^{\delta_0}}(x,y)&  \textrm{ if } x,y \in Y^{\delta_0},\\
&\\
d_{Y^{\delta_{i+1}}}(x,y)&\textrm{ if } x,y \in Y^{\delta_{i+1}}
\setminus \varphi_{\delta_{i+1}, \delta_i}\left(Y^{\delta_i}\right) ,\\
&\\
d_{Y^{\delta_{i+1}}}\left(x, \varphi_{\delta_{i+1},\delta_0}(y)\right) &
\textrm{ if } x \in Y^{\delta_{i+1}}\setminus 
\varphi_{\delta_{i+1}, \delta_{i}}\left(Y^{\delta_{i}}\right)\\
& \textrm{ for some } i \in \N \textrm{ and } y \in Y^{\delta_0}, \\
d_{Y^{\delta_{i+j+1}}}\left(x,
\varphi_{\delta_{i+j+1},\delta_{{i+1}}}(y)\right) & 
\textrm{ if } x \in Y^{\delta_{i+j+1}}\setminus 
\varphi_{\delta_{i+j+1}, \delta_{i+j}}\left(Y^{\delta_{i+j}}\right)\\
& \textrm{ and } y\in Y^{\delta_{i+1}}\setminus 
\varphi_{\delta_{i+1}, \delta_i}\left(Y^{\delta_i}\right)
\textrm{ for some } i,j \in \N 
\end{array}
\right.
\]
where we set
\be
\varphi_{\delta_{i+j}, \delta_i}=\varphi_{\delta_{i+j}, \delta_{i+j-1}}
\circ \cdots \circ \varphi_{\delta_{i+1}, \delta_{i}}.
\ee
This glued limit is not defined using an arbitrary collection of
isometric embeddings but rather only those achieved as in Theorem~\ref{thm-gluing-isoms}.

Furthermore, for all $\delta\in (0,\delta_0]$ 
there exists a subsequence $M^{\delta}_{j_k}$ 
which converges in Gromov-Hausdorf sense to a compact
metric space $Y^\delta$ and for any such $Y^\delta$ there exists an isometric embedding 
\be
F_\delta=F_{\delta, \{\delta_i\}}:Y^{\delta} \to Y.   
\ee
such that for the $\delta_i$ in our sequence we have
\be \label{check-justified}
F_{\delta_i}(Y^{\delta_i}) \subset F_{\delta_{i+1}}(Y^{\delta_{i+1}}) 
\ee
If $\beta_j$ is any sequence decreasing to $0$, then
\be\label{eqn-onto-Y-d}
Y=\bigcup_{i=1}^{\infty} F_{\beta_j}(Y^{\beta_j}).
\ee

We say that a sequence of open Riemannian manifolds, $M_j$, has a glued 
limit space, $Y$, if there exists a sequence of $\delta_i\to 0$, satisfying the
hypothesis of this theorem.   A ``completed glued limit" is defined to be
the metric completion of a glued limit space, $\bar{Y}$, and the ``boundary of a glued limit space" is
the set, $\bar{Y}\setminus Y$.  
\end{thm}

\begin{rmrk} \label{rmrk-one-spline}
In Example~\ref{ex-one-spline} for sufficiently large $\delta_{0}$
each limit $Y^{\delta}$ is a ball in Euclidean 2-dimensional space.
According to Definition~\ref{def-glue} the glued limit space of this sequence
is constructed by taking the disjoint union of the ball $Y^{\delta_{0}}$ with concentric annulus $Y^{\delta_0/i+1}\setminus Y^{\delta_0/i}$.
\end{rmrk}

\begin{rmrk}
Note that the definition of the glued limit space depends on the choice of 
$\delta_i$ and the isometric embeddings
in Theorem~\ref{thm-gluing-isoms}.   Even if one fixes the sequence $\delta_i\to 0$ the glued limit need not be unique. 
See Example~\ref{ex-nonunique}.
\end{rmrk}

\begin{proof}
We first prove that $d_{Y}$ is positive definite.
For the first and second cases of the definition of $d_{Y}$,
we immediately see that $d_{Y}(x,y)=0$ iff $x=y$. 
For the third and fourth cases, notice that 
$\varphi_{\delta_{i+j+1},\delta_{i+1}}(y)=
(\varphi_{\delta_{i+j+1},\delta_{i+j}} \circ  
\varphi_{\delta_{i+j},\delta_{i+1}})(y)$
then 
\be
\varphi_{\delta_{i+j+1},\delta_{i+1}}(y)
\in \varphi_{\delta_{i+j+1},\delta_{i+j}}(Y^{\delta_{i+j}}).
\ee
Thus $x \ne \varphi_{\delta_{i+j+1},\delta_{i+1}}(y)$
and $d_Y(x,y)= d_{Y^{\delta_{i+j+1}}}\left(x,
\varphi_{\delta_{i+j+1},\delta_{i+1}}(y)\right)\ne 0$.

Define $F_{\delta_{i}}:Y^{\delta_{i}} \to Y$ in the following way: 

\[ 
F_{\delta_i}(y)=\left\{
\begin{array}{ll}
y & i=1\\
& \\
y & i>1,\;\;\; y \in Y^{\delta_i}\setminus \varphi_{\delta_i, \delta_{i-1}}(Y^{\delta_{i-1}})\\
& \\
\varphi^{-1}_{\delta_i,\delta_0}(y) & i>1, \;\;\; \varphi^{-1}_{\delta_i,\delta_0}(y) \in Y^{\delta_0}\\
&\\
\varphi^{-1}_{\delta_i,\delta_j}(y) & i>1,\;\;\; \varphi^{-1}_{\delta_i,\delta_j}(y)
\in Y^{\delta_j}\setminus \varphi_{\delta_j, \delta_{j-1}}(Y^{\delta_{j-1}})\\
& \text{for some } j>1
\end{array}
\right.
\]

What we are doing in the third and fourth part of the definition of $F_{\delta_i}$ is the following. Suppose that $y \in Y^{\delta_{0}/i}$
then either 
\be
y\in Y^{\delta_i}\setminus  \varphi_{\delta_i, \delta_{i-1}}(Y^{\delta_{i-1}})
\ee
 or $y \in \varphi_{\delta_i, \delta_{i-1}}(Y^{\delta_{i-1}})$. In the latter case there exists $y_{i-1} \in Y^{\delta_{i-1}}$ such that $y=\varphi_{\delta_i, \delta_{i-1}}(y_{i-1})$. If $i-1>1$, either $y_{i-1} \in Y^{\delta_{i-1}}\setminus  \varphi_{\delta_{i-1}, 
\delta_{i-2}}(Y^{\delta_{i-2}})$ or $y_{i-1} \in \varphi_{\delta_{i-1}, \delta_{i-2}}(Y^{\delta_{i-2}})$. Proceeding in the same way, if necessary, 
we find $j$ such that there exists $y_{j} \in Y^{\delta_j} \setminus \varphi_{\delta_j,\delta_{j-1}}(Y^{\delta_{j-1}})$ 
when $j>1$ and $y_{j} \in Y^{\delta_j}$ when $j=1$, 
that satisfies $y=\varphi_{\delta_i, \delta_j}(y_{j})$. 

It is easy to see that 
\be\label{ImfFdelta0/i}
F_{\delta_i}(Y^{\delta_i})=Y^{\delta_0} \union \left( \,\disjointunion \,_{j=1}^{i-1} 
\left(Y^{\delta_{j+1}}\setminus  \varphi_{\delta_{j+1}, \delta_j}\left(Y^{\delta_j}\right)\right)\right),
\ee
and for $j<i$, 
\be
F_{\delta_j}=F_{\delta_i} \circ \varphi_{\delta_i,\delta_j}.
\ee

For arbitrary $\delta$, by Theorem~\ref{thm-gluing-isoms}, there
exists a subsequence, $M^\delta_{j_k}$, which converges in
the GH sense to a limit $Y^\delta$.
Let $F_\delta: Y^\delta\to Y$
\[ 
F_{\delta}=F_{ \{\delta, \{\delta_i\}\}}=\left\{
\begin{array}{ll}
F_{\delta_{0}} \circ \varphi_{\delta_{0},\delta} & \delta_{0} < \delta \\
F_{\delta_{i+1}} \circ \varphi_{\delta_{i+1},\delta} & \delta_{i+1} \leq \delta < \delta_i
\end{array}
\right.
\]
where $\varphi_{\delta_{0},\delta},\varphi_{\delta_{i+1},\delta} $ are given in Theorem~\ref{thm-gluing-isoms}.

 Observe that in the latter case of the definition of $F_{\delta}$, $F_{\delta_i}=F_{\delta} \circ \varphi_{\delta,\delta_i}$. This and the definition of $F_{\delta}$ gives:

\be\label{OrderedImF}
F_{\delta_i}(Y^{\delta_i}) \subset F_{\delta}(Y^{\delta}) \subset F_{\delta_{i+1}}(Y^{\delta_{i+1}}).
\ee

Now we have $\beta_j$ decreasing to $0$, there exists $N$
sufficiently large that $\beta_j \le \delta_0$, and for
all $j \ge N$ we have
$\exists i_j$ such that $\beta_j \in [\delta_{i+1}, \delta_i)$.
From (\ref{ImfFdelta0/i}) and (\ref{OrderedImF}), taking
$\delta=\beta_i$, we conclude that
\be
Y=  \bigcup_{j=N}^\infty F_{\beta_j}(Y^{\beta_j})
= 
\bigcup_{j=1}^\infty F_{\beta_j}(Y^{\beta_j})
\ee 
because $F_{\beta_0}(Y^{\beta_0})\subset F_{\beta_N}(Y^{\beta_N})$.

To prove that $F_{\delta}$ is an isometric embedding it is enough to prove
that for each $F_{\delta_i}$. $F_{\delta_{0}}$ is an isometric embedding
by definition of $Y$. For $F_{\delta_{i+1}}$ we must check
three cases. Let $x,y \in Y^{\delta_{i+1}}$. First case:
$x,y\in Y^{\delta_{i+1}}\setminus  \varphi_{\delta_{i+1}, \delta_i}(Y^{\delta_i})$ then $F_{\delta_{i+1}}(x)=x$ and
$F_{\delta_{i+1}}(y)=y$. Second case:
$x \in Y^{\delta_{i+1}}\setminus  \varphi_{\delta_{i+1}, \delta_i}(Y^{\delta_i})$ and $y \in \varphi_{\delta_{i+1}, \delta_i}(Y^{\delta_i})$. 
Thus, $F_{\delta_{i+1}}(y)=\varphi^{-1}_{\delta_{i+1},\delta_{i+1-j}}(y) \in Y^{\delta_{i+1-j}} \setminus \varphi_{\delta_{i+1-j},\delta_{i-j}}(Y^{\delta_{i-j}})$ for some $j$. So
\begin{align*}
d_Y(F_{\delta_{i+1}}(x), F_{\delta_{i+1}}(y))&=
d_{Y^{\delta_{i+1-j}}}(F_{\delta_{i+1}}(x), \varphi_{\delta_{i+1},\delta_{i+1-j}}(F_{\delta_{i+1}}(y)))\\
&=d_{Y^{\delta_{i+1}}}(x,y).
\end{align*}
Finally, if $F_{\delta_{i+1}}(x)=\varphi^{-1}_{\delta_{i+1},\delta_{i+1-k}}(x)$, $F_{\delta_{i+1}}(y)=\varphi^{-1}_{\delta_{i+1},\delta_{i+1-j}}(y)$. Suppose that $k \leq j$. Recall that  
$\varphi_{\delta_{i+1},\delta_{i+1-k}}
\circ \varphi_{\delta_{i+1-k},\delta_{i+1-j}}=\varphi_{\delta_{i+1},\delta_{i+1-j}}$, then

\begin{align*}
d_Y(F_{\delta_{i+1}}(x), F_{\delta_{i+1}}(y))&=d_{Y^{\delta_{i+1-k}}}(
F_{\delta_{i+1}}(x),
\varphi_{\delta_{i+1-k},\delta_{i+1-j}}(F_{\delta_{i+1}}(y)))
\\
&= 
d_{Y^{{\delta_{i+1}}}}(
\varphi_{\delta_{i+1},\delta_{i+1-k}}(F_{\delta_{i+1}}(x)),
\varphi_{\delta_{i+1},\delta_{i+1-j}}(F_{\delta_{i+1}}(y)))
\\
&= d_{Y^{\delta_{i+1}}}(x,y).
\end{align*}

 The triangle inequality follows from the above paragraphs. For $x,y,z \in Y$, find
$\delta$ such that $x,y,z \in F_{\delta}(Y^{\delta})$. The triangle inequality holds for the preimages of $x,y,z$ and since $F_{\delta}$ is an isometric embedding, it also
holds for $x,y,z$. 
\end{proof}

\subsection{Glued Limits within Gromov-Hausdorff Limits}

Recall that in Theorem~\ref{H-lim-gives-delta-lim} 
we proved that if a sequence of
open Riemannian manifolds, $M_j$, has a Gromov-Hausdorff 
limit, $X$, then subsequences of the inner regions, $M_j^\delta$,
have Gromov-Hausdorff limits.   Here we assume that the $M_j$
also have a (possibly empty) completed glued limit space 
as in Theorem~\ref{def-glue}.   We
prove that this completed glued limit space is unique and provide a 
precise description as to how to find this completed glued limit space
as a subset of the Gromov-Hausdorff limit
[Theorem~\ref{thm-uniquenessGluedLimit}].   

Note that the completed glued limit need not agree with the
Gromov-Hausdorff limit [Remark~\ref{rmrk-one-spline-2}].   
In fact we provide an example where the completed glued limit
space is empty [Remark~\ref{rmrk-disappear-GH}].

It should be emphasized that we must assume the $M_j$ have
a completed glued limit to obtain uniqueness.   It is possible
that a sequence $M_j$ has a Gromov-Hausdorff limit and that one
needs a subsequence to obtain a glued limit and that different 
subsequences provide different completed glued limits 
[See Remark~\ref{GH-limit-subseq-for-glued}].

\begin{thm}\label{thm-uniquenessGluedLimit}
Let $\{M_j\}$ be a sequence of open manifolds 
that converge
in Gromov-Hausdorf sense to a compact metric space $(X,d_X)$.
Suppose $Y$ is a glued limit space of the $\{M_j\}$
defined as in Theorem~\ref{def-glue}.  Then the completed
glued limit 
$\bar{Y}$ is isometric to the closure, $\bar{U}_{\{j_k\}}\subset X$, of
any limit's inner union,  ${U}_{\{j_k\}}\subset X$,
defined as in Definition~\ref{defn-GH-U} for any subsequence $j_k$.   
In particular any completed glued limit and the closure of
any of the limit's inner regions are isometric.   
\end{thm}

We do not claim all the limit's inner regions are the same subset of
$X$ and in fact this is not true even after taking a closure. 
They are only isometric to one another. See Example~\ref{ex-F}.

\begin{proof}
Let $Y$ be a glued limit space defined using Theorems~\ref{def-glue} 
and~\ref{thm-gluing-isoms} via a sequence of isometric
embeddings $\varphi_j$ of $M^{\delta_i}_j \subset M^{\delta_{i+1}}_j$
into a sequence of compact metric spaces $Z_i$ rather than a single 
compact metric space $Z$.   

Since we have assumed the original sequence of
Riemannian manifolds has a glued limit space $Y$ without requiring
a subsequence, the following spaces are isometric:
\be
Y_{\{\delta_i\},\{j_k\}} \tilde{=}Y_{\{\delta_i\}}\tilde{=}Y_{\{\delta_i\},\{j'_k\}}
\ee
for any pair of subsequences $\{j_k\}$ and $\{j_k'\}$.

Recall that Theorem~\ref{def-glue}
states that for each $\delta>0$ 
there exists an isometric embedding 
\be
F_\delta:Y^{\delta} \to Y 
\ee
and
\be
 Y=\bigcup_{i=1}^\infty F_{\delta_i}(Y^{\delta_i})
 \textrm{ with } F_{\delta_i}(Y^{\delta_i}) \subset F_{\delta_{i+1}}(Y^{\delta_{i+1}}).
\ee

Since $Y^\delta$ is the Gromov-Hausdorff limit of the
inner regions $M^\delta_j$, it is isometric to the limit of
the inner regions $Y^\delta(j_k)\subset U_{\{j_k\}}\subset X$
of Theorem~\ref{GH-lim-gives-delta-lim}.   
Note that we need a subsequence for each $\delta$ to produce the limit of the
inner regions.   We can produce
a diagonal subsequence (also denoted $\{j_k\}$) such that 
\be
Y^\delta(j_k)\subset U_{\{j_k\}}\subset X \textrm{ is defined for all }
\delta\in \{\delta_i\}.
\ee   
So we have
isometric embeddings, 
\be
\psi_{\delta_i}: F_{\delta_i}(Y^{\delta_i})\subset Y \to Y^\delta(j_k)\subset X.
\ee

Since $F_{\delta_i}(Y^{\delta_i}) \subset F_{\delta_{i+1}}(Y^{\delta_{i+1}})$
for each $i$ and any $h$ we may study the restriction
\be
\psi_{\delta_{i+h}}: F_{\delta_i}(Y^{\delta_i})\subset Y \to 
Y^{\delta_{i+h}}(j_k) \subset U_{\{j_k\}}\subset X.
\ee
Since $F_{\delta_i}(Y^{\delta_i})$ and $X$ are compact, we can find a
subsequence $h_k$ depending on $i$ which converges to a limit
isometric embedding:
\be
\psi_{i,\infty}: F_{\delta_i}(Y^{\delta_i})\subset Y \to \bar{U}_{\{j_k\}}\subset X.
\ee
We may do this for each $i$ and diagonalize the subsequences if we
wish.   Since $\psi_{\delta_{i+h}}$ is a restriction of $\psi_{\delta_{i+1+h}}$
we see that $\psi_{i,\infty}$ is a restriction of $\psi_{i+1,\infty}$.  Thus we
may define an isometric embedding
\be
\psi_\infty: Y \to \bar{U}_{\{j_k\}}\subset X.
\ee
Extending this we have an isometric embedding:
\be
\bar{\psi}_\infty: \bar{Y} \to \bar{U}_{\{j_k\}}.
\ee
Since $X$ is compact, $\bar{U}_{\{j_k\}}$ is compact and thus so is
$\bar{Y}$.

We need only construct an isometric embedding from
$\bar{U}_{\{j_k\}}$ to $\bar{Y}$ to prove that these spaces are isometric because they are compact metric spaces.   
We repeat the same trick
as above but now using the fact that we have isometries 
\be
F'_{\delta_i}: Y^{\delta_i}(j_k)\to F_{\delta_i}(Y^{\delta_i})\subset Y
\ee
and
\be
Y=\bigcup_{i=1}^\infty F'_{\delta_i}(Y^{\delta_i}(j_k)).
\ee
with
\be
F'_{\delta_i}(Y^{\delta_i}(j_k)) 
\subset F'_{\delta_{i+1}}(Y^{\delta_{i+1}}(j_k)). 
\ee
Since $Y^{\delta_i}(j_k) \subset Y^{\delta_{i+1}}(j_k)$,
we may study for each $i$ and any $h$ the restriction
\be
F'_{i+h}: Y^{\delta_i}(j_k) \to Y\subset \bar{Y}.
\ee
Since we have shown $\bar{Y}$ is compact, a subsequence 
converges for each $i$ (and we can diagonalize these subsequences),
so that we obtain isometric embeddings
\be
F'_{i,\infty}: Y^{\delta_i}(j_k) \to  \bar{Y}.
\ee 
Since $F'_{i,\infty}$ is a restriction of $F'_{i+1,\infty}$ we can define
an isometric embedding
\be
F'_\infty: U_{\{j_k\}} \to \bar{Y}.
\ee
This extends to an isometric embedding from $\bar{U}_{\{j_k\}}$
to $\bar{Y}$.   Since we have a pair of isometric embeddings between a pair
of compact metric spaces, these metric spaces are isometric.

\end{proof}

\begin{rmrk}\label{rmrk-one-spline-2}
It is possible that the completed glued limit is not the same as the
Gromov-Hausdorff limit.
Example~\ref{ex-one-spline} has a glued limit which is
an open disk in Euclidean
space, its completed glued limit is the closed disk 
while its Gromov-Hausdorff limit
is a disk with a line segment attached.   
\end{rmrk}

\begin{rmrk}\label{rmrk-disappear-GH}
The glued limit of a sequence of open
Riemannian manifolds may exist but be the empty set.   See for example
the sequence $M_{2j+1}$ in Example~\ref{ex-why-subseq-2}.
This sequence converges in the Gromov-Hausdorff sense
but $U$ is an empty set.   It only satisfies 
the conditions of Theorem~\ref{thm-gluing-isoms}
in a trivial way: for each $\delta>0$
there exists $N_\delta\in \N$ such that 
$M_j^\delta = \emptyset$ for all $j \ge N_\delta$.
\end{rmrk}

\begin{rmrk}\label{GH-limit-subseq-for-glued}
A sequence of $M_j$ which converges
in the Gromov-Hausdorff sense may not have 
a glued limit space.  In fact one may need to
take a subsequence to obtain a glued limit and
different subsequences might produce 
different glued limit spaces.   In 
Examples~\ref{ex-why-subseq-1}-~\ref{ex-why-subseq-3}, 
the subsequence
$M_{2j}$ has a completed glued limit space which is isometric to 
$\bar{U}_{\{2j\}}$ and the subsequence 
$M_{2j+1}$ has a completed glued limit space which is isometric
to $\bar{U}_{\{2j+1\}}$, but the sequence $M_j$ itself does
not have a glued limit space.   We thus see that the different
glued limits obtained using different subsequences 
are quite different.   In particular in Example~\ref{ex-why-subseq-2}
the completed glued limit of the $M_{2j}$ agrees with the Gromov-Hausdorff
limit of the $M_j$ while the completed glued limit of $M_{2j+1}$ is empty.
\end{rmrk}

\subsection{Glued Limit Spaces when there are no Gromov-Hausdorff limits}

In the setting of Theorem~\ref{thm-gluing-isoms}, the subsequence
of manifolds $M_j$ such that $M_j^\delta \GHto Y^\delta$ need
not have any Gromov-Hausdorff limit. Here we discuss
an old example and present two new
examples.

\begin{rmrk} \label{rmrk-many-splines}
In Example~\ref{ex-many-splines} which had increasingly many splines,
it was seen that the Gromov-Hausdorff
limit for the sequence $M_{j}$ described there does not exist. 
However the sequence $M^{\delta}_{j}$ converges to 
the metric completition of the annulus $Ann_{(0,0)}(1+\delta,2-\delta)$ with 
the flat metric, see Remark~\ref{rmrk-many-splines-delta}. 
Start with $\delta_{0} < 1/2$ then 
\be
Y^{\delta_{0}}=Ann_{(0,0)}[1+\delta_0,2-\delta_0],
\ee
and
\be
Y^{\delta_0/(i+1)}\setminus  \varphi_{\delta_0/(i+1), \delta_0/i}
\left(Y^{\delta_0/i}\right)=
A_1 \cup A_2
\ee
where
\begin{eqnarray}
A_1&=&Ann_{(0,0)}[1+\delta_0/(i+1),1+\delta_0/i \,) \textrm{ and } \\
A_2 &=& Ann_{(0,0)}(2-\delta_0/i,2-\delta_0/(i+1)].
\end{eqnarray}
Thus $Y=Ann_{(0,0)}(1,2)$ with the flat length metric. This glued limit space $Y$ is precompact.  
\end{rmrk}

A similar example, also constructed using flat $M_j \subset \E^2$
with no Gromov-Haudorff limit has converging $M_j^\delta$, and a glued limit
space which is a flat open manifold that is bounded but not 
precompact:

\begin{ex}\label{ex-decreasing-splines}
We define a flat open manifold with $j$ splines of decreasing width:
\begin{eqnarray}
M_j &=& U_j \cup V_j \textrm{ where }\\
U_j&=&\{(r\cos(\theta), r\sin(\theta)): \,\,
r<4+ \sin (4\pi^2/\theta), \, \theta\in (2\pi/ j , 2\pi]\} \\
V_j&=& \{(r\cos(\theta), r\sin(\theta)): \,\,
r<4, \, \theta\in (0,2\pi/ j]\}.
\end{eqnarray}
As in Example~\ref{ex-many-splines},
$(M_j, d_{M_j})$ have no Gromov-Haudorff limit because 
they have increasingly many splines.   
Unlike Example~\ref{ex-many-splines}, for any number
$N$, there exists $\delta_N$ sufficiently small that
$M_j^{\delta_N}$ has $N$ splines.  In fact, 
\be
(M_j^\delta, d_{M_j})\GHto (Y^\delta, d_Y)
\ee
where $Y^\delta$ is $\delta$ inner region of the flat open manifold:
\begin{eqnarray}
Y &=& \{(r\cos(\theta), r\sin(\theta)): \,\,
r<4+ \sin (4\pi^2/\theta), \, \theta\in (0 , 2\pi]\} 
\end{eqnarray}
Taking the identity maps to be the isometric embeddings, we see
that $Y$ is also a glued limit space for the $M_j$ even though it
is bounded but not precompact.
\end{ex}

Recall Example~\ref{ex-many-pages} of a sequence of surfaces
which have no Gromov-Hausdorff limit.  We now modify this
example to obtain a sequence of manifolds with boundary that have
no Gromov-Haudorff limit but whose $\delta$ inner regions have Gromov-Haudorff limits
and we construct the glued limit space and see that it is also
bounded and not precompact.   This glued limit space is not a manifold.

\begin{ex}\label{ex-many-pages-2}
Let
\be
X_j = \big( [0,1]\times[0,1]\big) \,\disjointunion \, \big( [0,1]\times [0,1/2] \big)
\,\disjointunion \, \cdots \,\disjointunion \, \big( [0,1]\times [0,1/2^j] \big)
\ee
be a disjoint union of spaces with taxicab metrics glued with a gluing
map $\psi(0,y)=(0,y)$.  One may think of $X_j$ as a book with $j$ pages
of decreasing height glued along a spine on the left.  Within $X_j$ 
choose sets $A_j$ to be the union of the top edges of each of the pages.
If we take surfaces $M_j$ as constructed in Proposition~\ref{prop-lattice}
they now have boundary, such that
\be
d_{GH}(M_j, X_j)\to 0 \textrm{ and } d_{GH}(M_j^\delta, X_j \setminus T_\delta(A_j) ) \to 0.
\ee
As in Example~\ref{ex-many-pages}, the $M_j$ have no GH
converging subsequence because the $X_j$ have no GH
converging subsequence.

Observe that there exists $k_\delta$ such that for all $j>k_\delta$,
\be
X_j\setminus T_\delta(A_j) =
\big( [0,1]\times[0,\delta)\big) \,\disjointunion \, \big( [0,1]\times [0,1/2-\delta] \big)
\,\disjointunion \, \cdots \,\disjointunion \, \big( [0,1]\times [0,1/2^{k_\delta}-\delta] \big).
\ee
Since this sequence does not depend on $j$, it clearly converges
in the Gromov-Hausdorff sense.   Thus $M_j^\delta$ converge
to the same Gromov-Hausdorff limit space.   In fact they converge to
$X_\infty \setminus T_\delta(A_\infty)$ where
\be
X_\infty = \big( [0,1]\times[0,1]\big) \,\disjointunion \, \big( [0,1]\times [0,1/2] \big)
\,\disjointunion \, \cdots \,\disjointunion \, \big( [0,1]\times [0,1/2^j] \big) \cdots
\ee
and $A_\infty$ is the union of the tops of all of these pages.
In fact, $X_\infty$ is the glued limit space.
\end{ex}

\subsection{A Glued Limit Space which is not Geodesic}

Here we present an example whose glued limit space is not
geodesic or even a length space (and neither is its metric completion):

\begin{ex}\label{glue-space-not-length}
In Euclidean space, $\E^2$, define 
\be
M_j=(-1,1)\times(-1,1)\setminus [-1/2,1/2]\times[0,1-1/j].
\ee
Then for $\delta<1/4$ there is $J=J(\delta)$ such that 
\be
M^{\delta}_j= \big((-1+\delta,1-\delta)\times(-1+\delta,-\delta)\big) \,\disjointunion \, \big((-1+\delta,-1/2-\delta)\times(-\delta,1-\delta)\big) \,\disjointunion \,
\big((1/2-\delta,1-\delta,)\times(-\delta,1-\delta)\big)\ee for $j \geq J$.

Thus $\bar{M}^{\delta}_j$ is a constant sequence for $j \geq J$ and 
$\bar{M}^{\delta}_j \GHto Y^\delta$
where 
\be Y^{\delta}=[-1+\delta,1-\delta]\times[-1+\delta,0]
\,\cup\, [-1+\delta,-1/2+\delta]\times[0,1-\delta]
\,\cup\, [1/2-\delta,1-\delta]\times[0,1-\delta].\ee
The completed glued limit
is not a length space:
\be
\bar{Y}=[-1,1]\times[-1,0]\,\cup\, [-1,-1/2]\times[0,1]\,\cup\, [1/2,1]\times[0,1] \subset \E^2
\ee 
Note that 
$\bar{M}_j \GHto X=Y \cup \big(\{1\}\times[-1/2,1/2]\big)$.
\end{ex}

\begin{question} \label{q-local-geod}
Is a glued limit space locally geodesic: for all $y\in Y$, does there
exist $\epsilon_y>0$ such that $B(y, \epsilon_y)$ is geodesic?   If
there is a counter example, what conditions can be imposed on the
space to guarantee that it is locally geodesic?
\end{question}

\subsection{Balls in Glued Limit Spaces}

Recall that earlier we proved that for any $p\in M^{\delta_i}$,
if $x\in B_p(\delta_i-\delta_{i+1}) \subset M$, then $x\in M^{\delta_{i+1}}$
[Lemma~\ref{ball-in-M-delta}].
This is not true for glued limit spaces.   That is, it is possible for
$p\in F_{\delta_i}(Y^{\delta_i})$ to have an 
$x\in B_p(\delta_i-\delta_{i+1}) \subset Y$ such that 
$x\notin F_{\delta_{i+1}}(Y^{\delta_{i+1}})$.   In fact we can take the
ball of arbitrarily small radius and still 
$x\notin F_{\delta_{i+1}}(Y^{\delta_{i+1}})$.

Here we present such an example:

\begin{ex}\label{ex-bad-balls}
Recall Example~\ref{ex-many-pages-2} where we constructed
$M_j$ that have no Gromov-Hausdorff limit such that
$M_j^\delta$ converge in the Gromov-Hausdorff sense to
$Y^\delta=X_\infty \setminus T_\delta(A_\infty)$ where
\be
X_\infty = \big( [0,1]\times[0,1]\big) \,\disjointunion \, \big( [0,1]\times [0,1/2] \big)
\,\disjointunion \, \cdots \,\disjointunion \, \big( [0,1]\times [0,1/2^j] \big) \cdots
\ee

where each piece is connected along $(0,y) \sim(0,y)$
and $A_\infty$ is the union of the tops of all of these pages.
This $X_\infty$ is a glued limit space for this example.

Then $F_\delta(Y^\delta)=X_\infty \setminus T_\delta(A_\infty)$.
Take any ball about the common point $(0,0)\in X_\infty$.
For any radius $r>0$, $B_{(0,0)}(r)$ contains infinitely many
points $y_j=(r/2,0)\in [0,1]\times[0,1/2^j]$.   However
$y_j \notin F_\delta(Y^\delta)$, for $j$ sufficiently large that 
$1/2^j < \delta$.   
\end{ex}

\subsection{Nonuniqueness of the Glued Limit Space}

We now see that glued limit spaces and completed glued limit spaces
are not necessarily unique.
Recall that in Theorem~\ref{thm-uniquenessGluedLimit} we explained that
if $M_j$ have a Gromov-Hausdorff limit, then the completed glued limit space
is unique.  So we need to construct a sequence of manifolds,
$M_j$, which have no Gromov-Hausdorff limit.   
In fact we will imitate Example~\ref{ex-many-pages-2}
applying Proposition~\ref{prop-lattice} to construct the
following example:

\begin{ex} \label{ex-nonunique}
There is a sequence of Riemannian
surfaces, $M_j$, with boundary, $\partial M_j$, such that
there exists $\delta_i \to 0$ and metric spaces $Y^{\delta_i}$
such that
\be
d_{GH}(M_j^{\delta_i}, Y^{\delta_i}) \to 0
\ee
yet there are two different glued limit spaces 
$Y_1=Y(\delta_{2i}, \varphi_{\delta_{2i}, \delta_{2i+2}})$ 
and $Y_2=Y(\delta_{2i}, \varphi'_{\delta_{2i}, \delta_{2i+2}})$ 
constructed as in Theorem~\ref{def-glue}
and Theorem~\ref{thm-gluing-isoms} and their metric
completions are not isometric.
\end{ex}

\begin{proof}
Let
\begin{eqnarray}
P_1 &=& [0,1]\times [-1/2,1/2] \\ 
P_2&=& [0,1]\times [-1/4,1/4]\\
P_3&=& [0,1]\times [-1/6, 1/6]\\
P_j&=&  [0,1]\times [-1/(2j), 1/(2j)]
\end{eqnarray}
and let
\begin{eqnarray}
X_j&=&P_1 \,\disjointunion \, P_2 \,\disjointunion \, P_2 \\
&& \,\disjointunion \, P_3 \,\disjointunion \, P_3 \,\disjointunion \, P_3 \,\disjointunion \, P_3 \\
&& \,\disjointunion \, P_4 \,\disjointunion \, P_4 \,\disjointunion \, P_4 \,\disjointunion \, P_4 
\,\disjointunion \, P_4 \,\disjointunion \, P_4 \,\disjointunion \, P_4 \,\disjointunion \, P_4 \\
&& \,\disjointunion \, P_5 \,\disjointunion \, \cdots \,\disjointunion \, P_5 \\
&&\,\disjointunion \, \cdots \,\disjointunion \, P_j \,\disjointunion \, \cdots \,\disjointunion \, P_j
\end{eqnarray}
be a disjoint union of $N_j=1+2+4+...+2^{j-1}$ spaces endowed with 
taxicab metrics glued with a gluing
map $\psi(0,y)=(0,y)$.  One may think of $X_j$ as a book with $N_j$ pages
of different heights glued along a spine on the left.  

Let $H_j \subset P_j$ be defined by
\begin{eqnarray}
H_j &=& [0,1]\times \{-1/2j\} \, \cup  \, \{1\} \times [-1/(2j),1/(2j)] 
  \, \cup \, [0, 1]\times \{1/(2j)\} \subset P_j
\end{eqnarray}
and let $A_j \subset X_j$ be defined,
\begin{eqnarray}
A_j&=&H_1 \,\disjointunion \, H_2 \,\disjointunion \, H_2 \\
&& \,\disjointunion \, H_3 \,\disjointunion \, H_3 \,\disjointunion \, H_3 \,\disjointunion \, H_3 \\
&& \,\disjointunion \, H_4 \,\disjointunion \, H_4 \,\disjointunion \, H_4 \,\disjointunion \, H_4 
\,\disjointunion \, H_4 \,\disjointunion \, H_4 \,\disjointunion \, H_4 \,\disjointunion \, H_4 \\
&& \,\disjointunion \, H_5 \,\disjointunion \, \cdots \,\disjointunion \, H_5 \\
&&\,\disjointunion \, \cdots \,\disjointunion \, H_j \,\disjointunion \, \cdots \,\disjointunion \, H_j.
\end{eqnarray}

If we take surfaces $M_j$ as constructed in Proposition~\ref{prop-lattice}
they now have boundary, such that
\be
d_{GH}(M_j, X_j)\to 0 \textrm{ and } d_{GH}(M_j^\delta, X_j \setminus T_\delta(A_j) ) \to 0.
\ee
As in Example~\ref{ex-many-pages}, the $M_j$ have no GH
converging subsequence because the $X_j$ have no GH
converging subsequence.

Now
\begin{eqnarray}
X_j\setminus T_\delta(A_j)&=&
(P_1\setminus T_\delta(H_1)) \,\disjointunion \, 
(P_2 \setminus T_\delta(H_2) )\,\disjointunion \, 
(P_2 \setminus T_\delta(H_2) )\\
&& \,\disjointunion \, (P_3\setminus T_\delta(H_3)) \,\disjointunion \, 
\cdots \,\disjointunion \, (P_3\setminus T_\delta(H_3)) \\
&&\,\disjointunion \, \cdots \,\disjointunion \, 
(P_j \setminus T_\delta(H_j))\,\disjointunion \, \cdots \,\disjointunion \, 
(P_j \setminus T_\delta(H_j) ).
\end{eqnarray}
Observe that 
\be
P_j \setminus T_\delta(H_j)
= [0, 1-\delta]\times [-1/(2j)\,+\delta, 1/(2j) \, -\delta]
\ee
 
Taking $\delta=\delta_{2i}=1/(2i)$ and $j>i$ we have
\be
P_{i} \setminus T_\delta(H_{i})
= [0, 1-1/(2i)]\times \{0\} 
\ee
and
\be
P_j \setminus T_\delta(H_j) =\emptyset
\ee
Thus
\begin{eqnarray}
X_j\setminus T_\delta(A_j)&=&
(P_1\setminus T_\delta(H_1)) \,\disjointunion \, 
(P_2 \setminus T_\delta(H_2) )\,\disjointunion \, 
(P_2 \setminus T_\delta(H_2) )\\
&& \,\disjointunion \, (P_3\setminus T_\delta(H_3)) \,\disjointunion \, 
\cdots \,\disjointunion \, (P_3\setminus T_\delta(H_3)) \\
&&\,\disjointunion \, \cdots \,\disjointunion \, 
(P_{i-1} \setminus T_\delta(H_{i-1})\,\disjointunion \, \cdots \,\disjointunion \, 
(P_{i-1} \setminus T_\delta(H_{i-1}) ) \\
&&\,\disjointunion \, 
 [0, 1-1/(2i)]\times \{0\} \,\disjointunion \, \cdots \,\disjointunion \,  [0, 1-1/(2i)]\times \{0\}
\end{eqnarray}
endowed with 
taxicab metrics glued with a gluing
map $\psi(0,y)=(0,y)$. 
There are $1+2+4+...+2^{(i-1)-1}$ rectangular pages 
and $2^{(i-1)}$ pages that are just
intervals of length $1-1/(2i)$.    Taking $j \to \infty$ we get
\be
d_{GH}(X_j\setminus T_\delta(A_j), Y^\delta) \to 0
\ee
where 
\be
Y^{\delta_{2i}}=Y^{1/(2i)}=X_j\setminus T_{\delta_{2i}}(A_j) \qquad \forall j>i.
\ee
So
\begin{eqnarray}
Y^{\delta_{2i}}&=&
(P_1\setminus T_{1/(2i)}(H_1)) \,\disjointunion \, 
(P_2 \setminus T_{1/(2i)}(H_2) )\,\disjointunion \, 
(P_2 \setminus T_{1/(2i)}(H_2) )\\
&& \,\disjointunion \, (P_3\setminus T_{1/(2i)}(H_3)) \,\disjointunion \, 
\cdots \,\disjointunion \, (P_3\setminus T_{1/(2i)}(H_3)) \\
&&\,\disjointunion \, \cdots \,\disjointunion \, 
(P_{2i-1} \setminus T_{1/(2i)}(H_{2i-1})\,\disjointunion \, \cdots \,\disjointunion \, 
(P_{2i-1} \setminus T_{1/(2i)}(H_{2i-1}) ) \\
&&\,\disjointunion \, 
 [0, 1-1/(2i)]\times \{0\} \,\disjointunion \, \cdots \,\disjointunion \,  [0, 1-1/(4i)]\times \{0\}
\end{eqnarray}
endowed with 
taxicab metrics glued with a gluing
map $\psi(0,y)=(0,y)$
where there are $1+2+4+...+2^{(i-1)-1}$ rectangular pages 
and $2^{(i-1)}$ pages that are just
intervals of length $1-1/(2i)$.

If we define $\varphi_{\delta_{2i}, \delta_{2i+2}}:
Y^{\delta_{2i}}\to Y^{\delta_{2i+2}}$ to be the inclusion map, and then
construct the glued
limit space as in Theorem~\ref{thm-gluing-isoms} we obtain,
\begin{eqnarray}
Y_1=Y(\delta_{2i}, \varphi_{\delta_{2i}, \delta_{2i+2}})=Y
&=&
(P_1\setminus H_1) \,\disjointunion \, 
(P_2 \setminus H_2 )\,\disjointunion \, 
(P_2 \setminus H_2 )\\
&& \,\disjointunion \, (P_3\setminus H_3) \,\disjointunion \, 
\cdots \,\disjointunion \, (P_3\setminus H_3) \cdots\\
&& \cdots \,\disjointunion \, (P_j\setminus H_j) \,\disjointunion \, 
\cdots \,\disjointunion \, (P_j\setminus H_j) \cdots
\end{eqnarray}
endowed with 
taxicab metrics glued with a gluing
map $\psi(0,y)=(0,y)$.
This has infinitely many pages, all shaped like rectangles.

Now we define $\varphi'_{\delta_{2i}, \delta_{2i+2}}:
Y^{\delta_{2i}}\to Y^{\delta_{2i+2}}$ to be an
isometric embedding which maps a point
\be
(x,y) \in P_k \setminus T_{\delta_{2i}}(H_k) 
\subset Y^{\delta_{2i}}
\ee
for $k< i$ to 
\be
(x,y) \in P_k \setminus T_{\delta_{2i+2}}(H_k) 
\subset Y^{\delta_{2i+2}}
\ee
via the inclusion map and which maps
\be
(x,y)\in P_{i} \setminus T_{\delta_{2i}}(H_{i})= [0, 1-1/(2i)]\times \{0\}
\subset Y^{\delta_{2i}}
\ee
to
\be\label{claim-claim}
(x, y-\delta_{2i}+\delta_{2i+2}) \in 
P_{i+1} \setminus T_{\delta_{2i+2}}(H_{i+1}) 
=[0, 1-1/(2i+2)]\times \{0\}
\subset Y^{\delta_{2i+2}}.
\ee
This is possible because we have enough copies
of $P_{i+1} \setminus T_{\delta_{2i+2}}(H_{i+1})$
in $Y^{\delta_{2i+2}}$.

In particular $\varphi'_{\delta_{2i}, \delta_{2i+2}}$ maps the
interval pages into interval pages.   If we then construct the glued
limit space as in Theorem~\ref{thm-gluing-isoms} we obtain,
\be
Y_2=Y(\delta_{2i}, \varphi'_{\delta_{2i}, \delta_{2i+2}})=
Y \,\disjointunion \,
 [0, 1]\times \{0\} \,\disjointunion \, 
 [0, 1]\times \{0\} \,\disjointunion \, 
 [0, 1]\times \{0\} \,\disjointunion \, 
 \cdots 
\ee
which has infinitely many pages that are intervals in addition to
all the pages shaped like rectangles.   
So we have two
distinct glued limit spaces for the sequence $\delta_{2i}=1/(2i)$
and their metric completions are not isometric.
\end{proof}

\section{Glued Limits under Curvature Bounds} \label{sect-curv}

In this section we prove the existence of glued limits of sequences
of manifolds with certain natural geometric conditions [Theorems~\ref{glued-constsec-limits} and~\ref{glued-Ricci-limits}].   We do
not require the sequences of manifolds themselves to have
Gromov-Hausdorff limits.

\subsection{Constructing Glued Limits of Manifolds with Constant Sectional Curvature}

In this section we prove that if $M_j \in \mathcal{M}^{m, V,l}_{H}$ (see
Definition~\ref{def-sc-const}) then the sequence has a glued limit space 
[Theorem~\ref{glued-constsec-limits}].  The sequence need not have
a Gromov-Hausdorff limit (see Remark~\ref{gcl-rmrk}).  

\begin{thm}\label{glued-constsec-limits}
Given any $\delta_0>0$ if  $(M_j, g_j) \subset \mathcal{M}^{m, V,l}_{H}$, then
there is a Gromov-Hausdorff convergent subsequence 
$\{M^{\delta_0}_{j_k}\}$ and a 
glued-limit space $Y$ such that for all $\delta\in (0,\delta_0]$ 
there exists a further subsequence $\{j_k'\}$ of $\{j_k\}$ for which
$M^{\delta}_{j_k'}$ converges in Gromov-Hausdorf sense to a compact
metric space $Y^\delta$ and for any such $Y^\delta$ there exists an isometric embedding 
\be
F_{\delta}:Y^{\delta} \to Y.
\ee
\end{thm}

\begin{rmrk} \label{gcl-rmrk}
The sequences  of flat surfaces, $M_j \subset \E^2$, defined in Example~\ref{ex-many-splines} and 
Example~\ref{ex-decreasing-splines} have a common finite upper volume bound but there is no common finite
upper bound for the number of disjoint balls of $M_j$ of radius less than 1. Thus, these two sequences
do not have a Gromov-Hausdorff limit. Nonetheless since 
\be
L_{min}(M_j)=\inf\{ L_g(C): \,\, C \textrm{ is a closed geodesic in $M_j$ } \} > l
\ee
Theorem~\ref{glued-constsec-limits} demonstrates that we can construct
glued limits for these spaces.
\end{rmrk}

\begin{rmrk}
The choice of the a further subsequence $\{j_k'\}$ of $\{j_k\}$ in Theorem~\ref{glued-constsec-limits} is necessary.
Let  $(M_j, g_j) \subset \mathcal{M}^{2, V,l}_{0}$ be the sequence defined in
defined in Example~\ref{ex-no-diag}.  
Take $\delta_0=3\varepsilon$. Then $\{M^{\delta_0}_j\}$ is a Gromov-Hausdorffconvergent sequence. 
Choosing $2\varepsilon \in (0,\delta_0]$ we see that
$\bar{M}^{2\varepsilon}_{2j}$ converges in Gromov-Hausdorff sense but 
$\bar{M}^{2\varepsilon}_{j}$
does not. 
\end{rmrk}

\begin{proof}
Consider the sequence $\delta_0$, $\delta_i=\delta_0/i$, $i \in \N$. Start with $\delta_{0}$. 
By Theorem~\ref{sc-const} there exist a subsequence $\{j_{k}(\delta_{0})\}_{k=1}^{\infty}$ of $\{j\}_{j=1}^{\infty}$
and a compact metric space $Y^{\delta_0}$ such that 
\be  
 \left(\bar{M}_  {{j_{k}(\delta_{0})}},d_{M_ {{j_{k}(\delta_{0})}}}\right)
 \GHto \left(Y^{\delta_0}, d_{Y^{\delta_0}}\right).
 \ee
Proceeding as before 
for $n \in \N$, there is a subsequence $\{j_{k}(\delta_{n})\}_{k=1}^{\infty}$ of
$\{j_{k}(\delta_{n-1})\}_{k=1}^{\infty}$
and a compact metric space $Y^{\delta_n}(j_{k}(\delta_{n}))$ such that 
\be  
(\bar{M}_{j_{k}(\delta_{n})}^{\delta_n},d_{M_ {{j_{k}(\delta_{n})}}}) \GHto Y^{\delta_n}.
\ee
Define $j_{k}=j_{k}(\delta_{k})$.  We have 
\be  
(\bar{M}_{j_{k}}^{\delta_n},d_{M_ {{j_{k}}}}) \GHto Y^{\delta_n}
\ee for $n=0,1,2,\cdots$ since $\{ j_{k} \}_{k=n}^{\infty}$ is a subsequence of $\{j_{k}(\delta_{n})\}_{k=1}^{\infty}$. 
We may now apply Theorem~\ref{def-glue} to complete the proof.
\end{proof}

\subsection{Constructing Glued Limits with Ricci curvature bounds}

Here we prove that glued limits exist for noncollapsing sequences of manifolds
with nonnegative Ricci curvature and bounded volume which have
control on the intrinsic diameters of their inner regions (defined in
(\ref{intrinsic-distance}):

\begin{thm}\label{glued-Ricci-limits}
Given $m\in \N$, a decreasing sequence, $\delta_i \to 0$, $i=0,1,2,\cdots$, $V>0$, $\theta>0$, 
and $D_i>0$, let $(M_j, g_j)$ be a sequence of $m$ dimensional open Riemannian
manifolds with nonnegative Ricci curvature, $\vol(M)\le V$, such that
\be
sup\Big\{\diam\left(M^{\delta_i}_{j}, {\,d_{M_j^{\delta_i}}}\right):\,\,j \in \N\Big\} < D_{i} \qquad
\forall i \in \N,
\ee
and 
\be \label{V-epsilon}
\forall j\in \N \,\, \exists q_j \in M_j^{\delta_0} \,\,\textrm {such that }
\vol(B_{q_j}(\delta_0)) \ge \theta\delta_0^m.
\ee

Then there exists a subsequence $j_k$ such that for all $\delta_i$
$\{M^{\delta_i}_{j_k}\}$ converge in the Gromov-Hausdorff sense
to a compact metric space $Y^{\delta_i}$.  Thus $M_{j_k}$ have a 
glued-limit space $Y$ such that for all $\delta\in (0,\delta_0]$ 
there exists a further subsequence $\{j_k'\}$ of $\{j_k\}$ for which
$M^{\delta}_{j_k'}$ converges in Gromov-Hausdorf sense to a compact
metric space $Y^\delta$ and for any such $Y^\delta$ there exists an isometric embedding 
\be
F_{\delta}:Y^{\delta} \to Y.
\ee
\end{thm}

\begin{rmrk}
If there is $D>0$ such that 
\be  
sup_{\delta \in (0,\delta_0]}
\Big\{\diam\big(M^{\delta}_{j}, {d_{M_j^\delta}}\big)\Big\} \leq D
\ee
Then we could take $D_i=D$ for all $i$. But this requirement is
unnecessarily strong.
\end{rmrk}

\begin{rmrk}
The choice of a further subsequence $\{j_k'\}$ of $\{j_k\}$ in 
Theorem~\ref{glued-constsec-limits} is necessary.
For the sequence $(M_j, g_j)$ defined in Example~\ref{ex-no-diag}, consider
a decreasing sequence, $\delta_i \to 0$, $i=0,1,2,\cdots$ such that 
$\delta_0=3\varepsilon$ and $\delta_1=\varepsilon$. Then the hypotheses of the
theorem are satisfied. For all $\delta_i$, $\{M^{\delta_i}_j\}$ converges in Gromov-Hausdorff sense. Now for $2\varepsilon \in (0,\delta_0]$, $\{M^{2\varepsilon}_j\}$ does not have a Gromov-Hausdorff limit.
\end{rmrk}

\begin{proof}
Take $\delta \in (0,\delta_0]$, 
by hypothesis and Bishop-Gromov volume comparison theorem~\ref{BGcomparison}
\be
\vol(B_{q_j}(\delta)) \geq \vol(B_{q_j}(\delta_0))\left( \frac{\delta}{\delta_0} \right)^m \geq \theta\delta^m.
\ee
The above inequality and the hypotheses of the theorem imply that for 
each $i$,
\be
\{(M_j, g_j)\} \subset \mathcal{M}^{m, \delta_i, D_i, V}_{\theta}.
\ee
Start with $\delta_{0}$. By Theorem~\ref{main-thm} there exists a subsequence $\{j_{k}(\delta_{0})\}_{k=1}^{\infty}$ of $\{j\}_{j=1}^{\infty}$ such that 
\be
 \Big(\bar{M}^{\delta_0}_ {{j_{k}(\delta_{0})}},d_{M_ {{j_{k}(\delta_{0})}}}\Big)
\GHto
\Big(Y^{\delta_0}, d_{Y^{\delta_0}}\Big).
\ee 
Proceeding as before for $n \in \N$, there exists a subsequence $\{j_{k}(\delta_{n})\}_{k=1}^{\infty}$ of
$\{j_{k}(\delta_{n-1})\}_{k=1}^{\infty}$
and a compact metric space $Y^{\delta_n}(j_{k}(\delta_{n}))$ such that 
\be  
\Big(\bar{M}_{j_{k}(\delta_{n})}^{\delta_n},d_{M_ {{j_{k}(\delta_{n})}}}\Big) \GHto \Big(Y^{\delta_n}, d_{Y^{\delta_n}}\Big).
\ee
Define $j_{k}=j_{k}(\delta_{k})$.  We have 
\be  
\Big(\bar{M}_{j_{k}}^{\delta_n},d_{M_ {{j_{k}}}}\Big) \GHto 
\Big(Y^{\delta_n}, d_{Y^{\delta_n}}\Big)
\ee
for $n\in \N$ since $\{ j_{k} \}_{k=n}^{\infty}$ is a subsequence of $\{j_{k}(\delta_{n})\}_{k=1}^{\infty}$. Finally, apply Theorem~\ref{def-glue}.
\end{proof}

\section{Properties of Glued Limit Spaces under Curvature Bounds}
\label{sect-prop}

In this final section of the paper we consider the local properties
of the glued limits of sequences of manifolds with constant
sectional curvature as in Theorem~\ref{glued-constsec-limits} 
and manifolds with
nonnegative Ricci curvature as in Theorem~\ref{glued-Ricci-limits}.   
We begin with an example indicating how even when the sequences of
manifolds has a Gromov-Hausdorff limit, one need not retain
curvature conditions on the Gromov-Haudorff limit space [Example~\ref{ex-to-negative}].   This is in sharp contrast 
with the setting where the Riemannian manifolds are compact
without boundary.   In this example, the glued limit space is empty.
Then we have a subsection about balls in
glued limit spaces without any assumption on curvature
[Theorem~\ref{thm-balls-in-limits-1}].  
We apply this control on the balls to prove that local curvature
properties do persist on glued limit space.   In particular we prove
Proposition~\ref{sc-prop} that the glued limits of manifolds with
constant sectional curvature bounds (and other conditions)
are unions of manifolds with constant sectional curvature.    
We close with Theorem~\ref{thm-ricci-glued-lim-measure}
concerning the metric measure properties of glued limits of
manifolds with nonnegative Ricci curvature.

\subsection{An Example with no Curvature Control}

We now construct a sequence of flat open manifolds whose
Gromov-Hausdorff limit is not flat:

\begin{ex}\label{ex-to-negative}
Let $B_p(1)\subset \HH^2$ be a unit ball in hyperbolic space
and $B_0(1) \subset \E^2$ be the unit ball in Euclidean space.
Then $\exp_p: B_0(1) \to B_p(1)$.   Let 
\be
S_j =\{ ( i/j, k/j ): \, i, k \in \Z \, \} \cap B_0(1) \subset \E^2
\ee
and $S'_j = \exp_p(S_j)$.   We can form a graph $A_j$
whose vertices are in $S_j$ and whose edges form a triangulation.
That is we connect $(i/j, k/j)$ to the points $((i+1)/j, k/j)$,
$(i/j, (k+1)/j)$ and $((i+1)/j, (k+1)/j)$.   We let $A'_j=\exp_p(A_j)$
and set the lengths of the edges in $A'_j$ to be the
distances between the vertices viewed as points in $\HH^2$.
Then $A'_j$ converges to $B_p(1)\subset \HH^2$ in the Gromov-Hausdorff
sense.

Now we define $A''_j$ to be the simplicial complexes formed by
filling in the triangles in $A'_j$ with flat Euclidean triangles.
Observe that $A''_j$ converges to $B_p(1)\subset \HH^2$ in the Gromov-Hausdorff sense as well.   Finally, for each $j$ we remove tiny balls
of radius $<< 1/j$ 
around the vertices in $A''_j$, to create a flat open manifold, $M_j$.
These $M_j$ converge in the Gromov-Hausdorff sense to
$B_p(1) \subset \HH^2$.
\end{ex} 

\begin{rmrk}
Example~\ref{ex-to-negative} 
has an empty glued limit space.   In the next subsections
we will see that the glued limit spaces do retain some of the
curvature properties of the initial sequence of manifolds.   Thus
the glued limit space is a more natural object of study than the Gromov-Haudorff limit
even when the Gromov-Haudorff limit exists.
\end{rmrk}

\subsection{Balls to Glued Limit Spaces}

Generally when one wishes to study the properties
of a complete noncompact limit space, one studies balls
in the limit space as Gromov-Hausdorff limits of balls
in the sequence.   Here we cannot control balls in the
limit space, but we can control balls of radius
$\epsilon< \delta_i-\delta_{i+1}$ centered in 
$F_{\delta_i}(Y^{\delta_i})$ 
intersected with $F_{\delta_{i+1}}(Y^{\delta_{i+1}})$.
This will suffice to study the geometric properties of the
glued limit spaces.

\begin{thm}\label{thm-balls-in-limits-1}
Let $Y$ be a glued limit of a sequence of Riemannian manifolds
$M_j$ as in Theorem~\ref{def-glue}.
If $y \in Y^{\delta_i}$ and $\epsilon< \delta_i-\delta_{i+1}$
, then there exists a subsequence $M^{\delta_i}_{j_k}$ 
containing points $y_{j_k}$ and $\epsilon_{j_k} \to \epsilon$ such that 
\be\label{ball-in-good}
B(y_{j_k},\epsilon_{j_k})=\Big\{x\in M_{j_k}: d_M(x, y_{j_k})<\epsilon_{j_k}\Big\} \subset
M_{j_k}^{\delta_{i+1}}.
\ee
and 
\be \label{GH-ball-in-good}
d_{GH}\Big( \left(\bar{B}(y_{j_k}, \epsilon_{j_k}), d_{M_{j_k}}\right), 
\left(\bar{B}(F_{\delta_i}(y),\epsilon)\cap F_{\delta_{i+1}}(Y^{\delta_{i+1}}), d_{Y}\right)\Big)\to 0.
\ee
\end{thm}

Note that in Example~\ref{ex-bad-balls} we saw that 
$
B(F_{\delta_i}(y),\epsilon)\cap F_{\delta_{i+1}}(Y^{\delta_{i+1}})
$
need not be isometric to $B(F_{\delta_i}(y),\epsilon)\subset Y$
even when $\epsilon$ is taken arbitrarily small.

\begin{proof}
Recall that $\varphi_{\delta_{i+1},\delta_i}$ was defined in the following way,
see Theorem\ref{thm-gluing-isoms}. 
We picked isometric embeddings 
\be 
\varphi_{j}:M^{\delta_{i+1}}_j\to Z,
\ee 
\be
\varphi_{\infty}:Y^{\delta_{i+1}}\to Z
\ee such that 
\be
d_{H}^{Z}\Big(\varphi_j(M^{\delta_{i+1}}_j),\varphi_\infty(Y^{\delta_{i+1}}_j)\Big)\to 0
\ee
Then we found a subsequence such that
\be
d_{H}^{Z}\Big(\varphi_{j_k}(M^{\delta_i}_{j_k}),X^{\delta_i}\Big)\to 0
\ee
and chose $\varphi_{\delta_{i+1},\delta_i}$ to be an isometry such that
\be
\varphi_{\delta_{i+1},\delta_i}(Y^{\delta_i})=\varphi_{\infty}^{-1}(X^{\delta_i}).
\ee

Then there exist
\be
y_{j_k}\in M_{j_k}^{\delta_i}\subset M_{j_k}^{\delta_{i+1}}\subset M_{j_k}
\ee
such that 
\be  
d_Z\Big(\varphi_{j_k}(y_{j_k}),\varphi_{\infty}(\varphi_{\delta_{i+1},\delta_i}(y)\Big)\to 0
\ee

Let $\epsilon'\in (0, \delta_i-\delta_{i+1})$.
Then by Lemma~\ref{ball-in-M-delta} we have
\be
B(y_{j_k},\epsilon')=\Big\{x\in M_{j_k}: d_M(x, y_{j_k})<\epsilon'\Big\} \subset
M_{j_k}^{\delta_{i+1}}.
\ee
From this and since $\varphi_{j_k}:M^{\delta_{i+1}}_{j_k} \to Z$ is an isometry into its image:
\be
\Big(B(y_{j_k},\epsilon'), d_{M^{\delta_{i+1}}_{j_k}}\Big)
\,\,\textrm{ is isometric to }\,\,
\Big(B(\varphi_{j_k}(y_{j_k}),\epsilon') \cap \varphi_{j_k}(M^{\delta_{i+1}}_{j_k}),d_Z\Big).
\ee
By Lemma~\ref{lem-H-balls}, for any $\epsilon\in (0, \delta_i-\delta_{i+1})$, 
there exists $\epsilon_{j_k}\to \epsilon$ eventually in $(0, \delta_i-\delta_{i+1})$, such that
\be  
\bar{B}\Big(\varphi_{j_k}(y_{j_k}),\epsilon_{j_k}\Big) 
\cap \varphi_{j_k}\Big(M^{\delta_{i+1}}_{j_k}\Big)\,\, \Hto\,\,
\bar{B}\Big(\varphi_{\infty}\varphi_{\delta_{i+1},\delta_i}(y),\epsilon\Big) \cap \varphi_{\infty}\Big(Y^{\delta_{i+1}}\Big).
\ee
Now,
\be
\Big(\bar{B}(\varphi_\infty\varphi_{\delta_{i+1},\delta_i}(y),\epsilon) \cap \varphi_\infty(Y^{\delta_{i+1}}), d_Z\Big)
\ee
 is isometric to 
 \be
\Big(\bar{B}(\varphi_{\delta_{i+1},\delta_i}(y),\epsilon), d_{Y^{\delta_{i+1}}}\Big)
\ee
which is isometric to 
\be
\Big(F_{\delta_{i+1}}\bar{B}(\varphi_{\delta_{i+1},\delta_i}(y),\epsilon), d_{F_{\delta_{i+1}}(Y^{\delta_{i+1}})}\Big)
\ee
which is isometric to 
\be
\Big(\bar{B}(F_{\delta_{i+1}}\varphi_{\delta_{i+1},\delta_i}(y),\epsilon) \cap F_{\delta_{i+1}}Y^{\delta_{i+1}}, d_{Y}\Big).
\ee
Hence
\be
d_{GH}\Big( \Big(\bar{B}(y_{j_k}, \epsilon_{j_k}), d_{M_{j_k}}\Big), 
\Big(\bar{B}(F_{\delta_{i+1}}(\varphi_{\delta_{i+1},\delta_i}(y)),\epsilon)\cap F_{\delta_{i+1}}Y^{\delta_{i+1}}, d_{Y}\Big)\Big)\to 0.
\ee
\end{proof}

\subsection{Properties of Glued Limits of Manifolds with Constant Sectional Curvature} 

Here we prove a proposition, present a key example
and state two open questions concerning the glued limits of
manifolds with constant sectional curvature.

\begin{prop} \label{sc-prop}
Let $Y$ be a glued limit space obtained from a sequence
$M_j \in \mathcal{M}^{m, V, l}_{H}$
as in Theorem~\ref{glued-constsec-limits}.  Then there
exists a countable collection of sets, $W_i \subset Y$, 
each of which is isometric to an $m$ dimensional smooth open manifold
of constant sectional curvature, $H$,
such that 
\be \label{W-exhaust}
Y \subset \bigcup_{i=1}^\infty W_i.
\ee
In fact
\be
F_{\delta_i}(Y^{\delta_i}) \subset W_i \subset F_{\delta_{i+1}}(Y^{\delta_{i+1}})\subset Y.
\ee
\end{prop}

See Example~\ref{sc-glued-lim-planes} in which the glued limit space
is a countable collection of flat tori which are not connected
to one another but have a metric restricted from a larger compact
metric space of finite volume.

\begin{proof}
Recall that any glued limit space, $Y$, defined as in Theorem~\ref{def-glue}
depends on a sequence $\delta_i\to 0$ and gluings
$\varphi_{\delta_{i+1}, \delta_i}: Y^{\delta_{i+1}}\to Y^{\delta_i}$
via the subsequential
limit isometric embeddings of (\ref{subseqlim}).
There are isometric
embeddings $F_{\delta_{i}}:Y^{\delta_{i}}\to Y$ such that
\be \label{F-exhaust}
Y \subset \bigcup_{i=1}^\infty F_{\delta_{i}}(Y^{\delta_{i}})
\ee
and 
\be \label{F-inside}
F_{\delta_i}(Y^{\delta_i}) \subset  F_{\delta_{i+1}}(Y^{\delta_{i+1}}).
\ee
Let 
\be
\epsilon_i=\min\left\{\delta_i-\delta_{i+1}, l/2, \pi \sqrt{h}/2\right\}/2
\ee
where $h=H$ when $H>0$ and $h=(l/\pi)^2$ otherwise.

Let
\be
W_i=T_{\epsilon_i}\Big(F_{\delta_i}\big(Y^{\delta_i}\big)\Big)\cap F_{\delta_{i+1}}\big(Y^{\delta_{i+1}}\big)
 \subset Y
\ee 

First observe that by (\ref{F-inside}) we have 
\be
F_{\delta_i}(Y^{\delta_i}) \subset W_i.
\ee
So combined with (\ref{F-exhaust}), we have (\ref{W-exhaust}).
So we need only show $W_i$ is a smooth $m$ dimensional
open manifold of constant sectional curvature, $H$.

For all $w\in W_i$, there exists $y_\infty \in F_{\delta_i}(Y^{\delta_i})$
such that $w\in B_{y_\infty}(\epsilon_i)\subset Y$.   
Since 
\be
B_{y_\infty}(\epsilon_i)\subset T_{\epsilon_i}\Big(F_{\delta_i}\big(Y^{\delta_i}\big)\Big)
\ee
we have
\be \label{need-U}
U= B_{y_\infty}(\epsilon_i) \cap F_{\delta_{i+1}}(Y^{\delta_{i+1}})
=B_{y_\infty}(\epsilon_i) \cap W_i.
\ee
We need only show that $U$
is isometric to a ball of radius $\epsilon_i$ in 
$M_H^m$, the $m$ dimensional simply
connected manifold with constant sectional curvature $H$.

There exists $y\in Y^{\delta_i}$ such that
$y_\infty =F_{\delta_i}(y)$.
By Theorem~\ref{thm-balls-in-limits-1}, and the fact that
$\epsilon_i< \delta_i-\delta_{i+1}$, there 
exists a subsequence $M^{\delta_i}_{j_k}$ 
containing points $y_{j_k}$ and $\epsilon_{j_k} \to \epsilon_i$ such that 
(\ref{ball-in-good}) and (\ref{GH-ball-in-good}) are satisfied.     

Since $\epsilon_i<l/2$, then for $k$ sufficiently large
$\epsilon_{j_k}<l/2$ and so by (\ref{ball-in-good}) and
$M_j$ satisfy the conditions of Theorem~\ref{sc-const},
by (\ref{is-a-const-ball})
we know
there is an Riemannian isometric diffeomorphism from
$B(y_{j_k},\epsilon_{j_k})$ to a ball in $M_H^m$, the $m$ dimensional simply
connected manifold with constant sectional curvature $H$. 
Since $\epsilon_i < \sqrt{H}\pi/2$ when $H>0$, we have
a convex ball, so that, as metric spaces, 
\be
\big(B(y_{j_k},\epsilon_{j_k}), d_M\big)
\textrm{ is isometric to }
\big(B(p,\epsilon_{j_k}), d_{M_H^m}\big).
\ee  

Taking
$k\to \infty$, the closure of these latter balls converge in the Gromov-Hausdorff
sense to $(\bar{B}(p,\epsilon_i), d_{M_H^m})$.   Thus by
(\ref{GH-ball-in-good}) and the uniqueness of Gromov-Hausdorff limits,
\be \label{sc-GH-ball-in-good}
\big(\bar{B}(y_\infty,\epsilon_i)\cap F_{\delta_{i+1}}(Y^{\delta_{i+1}}), d_Y\big)
\textrm{ is isometric to }
\big(\bar{B}(p,\epsilon_i), d_{M_H^m}\big).
\ee
Thus we have (\ref{need-U}) and we are done.
\end{proof}

\begin{ex}\label{sc-glued-lim-planes}
In this example we construct a glued limit space, $Y$, for 
a sequence of manifolds, $M_j^m$, satisfying the conditions of 
Theorem~\ref{sc-const}.  In addition the $M_j^m$ converge
in the GH sense to a metric space $X$, so that the glued limit
space is unique.  The glued limit $Y$ is a countable union 
of connected flat manifolds with the restricted metric from $X$.
\end{ex}

\begin{proof}
Let $M_1$ be two flat square annuli connected by a slanted strip of 
width $1$ and
length $\sqrt{2}$:
\be
M_1=C_{0,1} \cup C_{1,1} \cup S_{0,1} \subset \R^3
\ee
where
\begin{eqnarray}
C_{0,1}&=& 
\Big( \big((-1,1)\times (-1,1)\big) \setminus 
\big((-1/2, 1/2) \times [-1/2,1/2] \big) \Big)
\times\{0\}\\
C_{1,1} &=& 
\Big( \big( (-1,1)\times (-1,1) \big) \setminus 
\big((-1/2, 1/2) \times [-1/2,1/2] \big) \Big)
\times\{1\}\\
S_{0,1} &=& \Big\{(x,y,z): \,\, (x,y)\in (-1/2, 1/2)\times [-1/2, 1/2], 
\,\, z=x+1/2 \Big\}.
\end{eqnarray}
Endowed with the length metric, 
this is isometric to an open manifold with constant sectional curvature 0.
Note that for $\delta>1/4$, 
\be
M_1^\delta\subset C_{0,1}\cup C_{1,1}.
\ee

Let $M_2$ be three flat square annuli of total area $\le 4+4+ 4(1/4)$
connected by two slanted strips of width $1/2$:
\be
M_2=C_{0,2} \cup C_{1,2}\cup C_{2,2} \cup S_{1,2}\cup S_{2,2} \subset \R^3
\ee
where
\begin{eqnarray}
C_{0,2}&=& 
\Big( \big((-1,1)\times (-1,1)\big) \setminus 
\big((-1/4, 1/4) \times [-1/4,1/4] \big) \Big)
\times\{0\}\\
C_{1,2} &=& 
\Big( \big( (-1/2,1/2)\times (-1/2,1/2) \big) \setminus 
\big((-1/4, 1/4) \times [-1/4,1/4] \big) \Big)
\times\{1/2\}\\
C_{2,2} &=& 
\Big( \big( (-1,1)\times (-1,1) \big) \setminus 
\big((-1/4, 1/4) \times [-1/4,1/4] \big) \Big)
\times\{2/2\}\\
S_{0,2} &=& \Big\{(x,y,z): \,\, (x,y)\in (-1/4, 1/4)\times [-1/4, 1/4] 
\,\, z=x+1/4 \Big\}\\
S_{1,2} &=& \Big\{(x,y,z): \,\, (x,y)\in (-1/4, 1/4)\times [-1/4, 1/4] 
\,\, z=x+3/4 \Big\}.
\end{eqnarray}
Endowed with the length metric, 
this is isometric to an open manifold with constant sectional curvature 0.
Note that for $\delta>1/8$, 
\be
M_1^\delta\subset C_{0,2}\cup C_{1,2}\cup C_{2,2} \setminus
\big( B((0,0), \delta) \times [0,1] \big).
\ee

Let $M_j$ be $(j+1)$ flat square annuli of total area $\le 4+4\sum_{i=0}^j(1/2)^j$
connected by j slanted strips of width $(1/2^j)$:
\be
M_j=\bigcup_{i=0}^j C_{i,j} \cup \bigcup_{i=0}^{j-1} S_{i,j}\subset \R^3
\ee
where
\begin{eqnarray*}
C_{0,j}&=& 
\Big( \big((-1,1)\times (-1,1)\big) \setminus 
\big((-1/2^{j+1}, 1/2^{j+1}) \times [-1/2^{j+1},1/2^{j+1}] \big) \Big)
\times\{0\}\\
C_{i,j} &=& 
\Big( \big( (-2^{i-j},2^{i-j})\times (-2^{i-j},2^{i-j}) \big) \setminus 
\big((-1/2^{j+1}, 1/2^{j+1}) \times [-1/2^{j+1},1/2^{j+1}] \big) \Big)
\times\{2^{i-j}\}\\
C_{j,j} &=& 
\Big( \big( (-1,1)\times (-1,1) \big) \setminus 
\big((-1/2^{j+1}, 1/2^{j+1}) \times [-1/2^{j+1},1/2^{j+1}] \big) \Big)
\times\{2^{j-j}\}\\
S_{0,j} &=& \Big\{(x,y,z): \,\, (x,y)\in (-1/2^{j+1}, 1/2^{j+1})\times [-1/2^{j+1}, 1/2^{j+1}] \,\, z=x+1/2^{j+1} \Big\}\\
S_{i,j} &=& \Big\{(x,y,z): \,\, (x,y)\in (-1/2^{j+1}, 1/2^{j+1})\times [-1/2^{j+1}, 1/2^{j+1}] \,\, z= m_j(x+1/2^{j+1}) +2^{i-j} \Big\}\\
\end{eqnarray*}
where
\be
m_j=\frac{2^{i+1-j}-2^{i-j}}{(1/2^{j+1}) - (-1/2^{j+1})}.
\ee 
Endowed with the length metric, 
this is isometric to an open manifold with constant sectional curvature 0.
Note that for $\delta>1/2^{j+2}$, 
\be
M_j^\delta\subset C_{0,j}\cup \cdots \cup C_{j,j} \setminus
\big( B((0,0), \delta) \times [0,1] \big).
\ee
The Gromov-Hausdorff limit of the $M_j$ exists and can be see to
be
\be
X= \bigcup_{j=0}^\infty C_j \cup S_0 \subset \R^3
\ee
where
\begin{eqnarray}
C_{0}&=& 
\Big( \big((-1,1)\times (-1,1)\big) \Big)
\times\{0\}\\
C_{i} &=& 
\Big( \big( (-2^{i-j},2^{i-j})\times (-2^{i-j},2^{i-j}) \big) \Big)
\times\{2^{i-j}\}\\
S_{0} &=& \Big\{(0,0,z): \,\, z\in [0,1]\Big\}
\end{eqnarray}
endowed with the length metric.
The Gromov-Hausdorff limit, $Y^\delta$ of the $M_j^\delta$ exists and
\be
Y^\delta \subset  X \setminus
\big( B((0,0), \delta) \times [0,1] \big).
\ee
In fact $Y= X \setminus S_0$.
\end{proof}

\begin{question}\label{q-sc-manifold}
Are the glued limits of sequences of manifolds with constant
sectional curvature open manifolds with constant sectional
curvature?   We know they need not be connected by
Example~\ref{sc-glued-lim-planes}.
\end{question}

\begin{question} \label{q-sc-unique}
Are the glued limits of sequences of manifolds with constant
sectional curvature unique?   Perhaps an adaption of
Example~\ref{sc-glued-lim-planes} could be applied to show that
they are not.   
\end{question}

\subsection{Properties of Glued Limits of Manifolds with Nonnegative Ricci Curvature}

We now prove the final theorem of our paper
and state the last two open questions:

\begin{thm}\label{thm-ricci-glued-lim-measure}
Suppose we have a sequence of $m$ dimensional
open Riemannian manifolds
$M_j$ with nonnegative Ricci curvature and $\vol(M_j)\le V_0$
and that there exists a
sequence $\delta_i \to 0$, such that the inner regions, 
$M_j^{\delta_i}$, converge in the Gromov-Hausdorff sense
as $j\to \infty$ to $Y^{\delta_i}$ without collapsing.   Suppose that $Y$ is
a glued limit constructed as in Theorem~\ref{def-glue}.
Then $Y$ has Hausdorff dimension $m$, $\mathcal{H}^m(Y)\le V_0$
and its Hausdorff measure has positive lower density everywhere.
\end{thm}

Note that this theorem may be applied to study the glued limits of
sequences of manifolds satisfying the conditions of 
Theorem~\ref{glued-Ricci-limits}.   

To prove this theorem we will apply Cheeger-Colding's Volume Convergence
Theorem \cite{ChCo-PartI}\cite{Colding-volume} which was reviewed in Subsection~\ref{subsect-vol-conv}.   See Theorem~\ref{thm-chco}
and Remark~\ref{rmrk-chco} for the precise statement we will use here.

\begin{proof}
First we prove that 
\be
W_i\,\,=\,\,
T_{(\delta_i-\delta_{i+1})/2}\left(F_{\delta_i}\left(Y^{\delta_i}\right)\right)
\,\,\cap\,\, F_{\delta_{i+1}}\left(Y^{\delta_{i+1}}\right)\,\,\subset\,\, Y.
\ee 
have Hausdorff dimension $m$ and have
 doubling Hasudorff measures.   For any $w\in W_i$,
 let
 \be
 U_w=B(w,(\delta_i-\delta_{i+1})/2) \cap W_i.
 \ee
 We can find $y \in Y^{\delta_i}$ 
 such that $d_Y(y, w) <(\delta_i-\delta_{i+1})/2$.
Then we have
\be\label{U-w}
U_w=
B\left(w,(\delta_i-\delta_{i+1})/2\right) 
\cap B\left( F_{\delta_i}(y), \delta_i-\delta_{i+1}\right)
\cap F_{\delta_{i+1}}(Y^{\delta_{i+1}})
\ee
By Theorem~\ref{thm-balls-in-limits-1}, we have
a subsequence $j_k$, points $y_{j_k}\in M_{j_k}^{\delta_i}$
and $\epsilon_{j_k}\to \epsilon=(\delta_i-\delta_{i+1})/2$
satisfying 
(\ref{ball-in-good}) and (\ref{GH-ball-in-good}):
\be 
d_{GH}\left( \left(\bar{B}(y_{j_k}, \epsilon_{j_k}), d_{M_{j_k}}\right), 
\left(\bar{B}(F_{\delta_i}(y),\epsilon)\cap F_{\delta_{i+1}}(Y^{\delta_{i+1}}), d_{Y}\right)\right)\to 0.
\ee
Combining this with the fact that
\be
w\in B(y, (\delta_{i}-\delta_{i+1})/2)\subset 
\bar{B}(F_{\delta_i}(y),\epsilon)\cap F_{\delta_{i+1}}(Y^{\delta_{i+1}})\subset
Y
\ee
there exists
\be
z_{j,k} \in \bar{B}(y_{j_k}, (\delta_{i}-\delta_{i+1})/2) \subset 
\bar{B}(y_{j_k}, \epsilon_{j_k})\subset
M_{j_k}
\ee
such that
\be \label{GH-ball-in-good-8}
d_{GH}\Big( \left(\bar{B}(z_{j_k}, (\delta_{i}-\delta_{i+1})/2 ), d_{M_{j_k}}\right), 
\left(\bar{U}_w, d_{Y}\right) \Big) \to 0.
\ee
Since we assumed this is noncollapsing, then by the Cheeger-Colding
Volume Convergence Theorem mentioned above we have
\be\label{colding-volume}
\mathcal{H}_m( B_w(r) \cap U_w)=\lim_{k\to \infty} \mathcal{H}_m(B_{z_{j_k}}(r))
\ee
for all 
$
r\le r_i=(\delta_i-\delta_{i+1})/2.
$
By (\ref{U-w})  and Bishop's Volume Comparison Theorem,
we see that
\be \label{HmWir}
\mathcal{H}_m( B_w(r) \cap W_i)=
\mathcal{H}_m( B_w(r) \cap U_w) \le \omega_m r^m\qquad \forall r\le r_i
\ee
is positive and finite for any $w\in W_i$.   By Bishop-Gromov's
Volume Comparison Theorem,
\be\label{Bishop-Gromov}
\frac{\mathcal{H}_m( B_w(r_1) \cap W_i)}{\mathcal{H}_m( B_w(r_2)\cap W_i)}
\ge \frac{r_1^m}{r_2^m} \qquad \forall \,\,w\in W_i,\,\, r_1<r_2\le r_i.
\ee
Since $W_i$ is a subset of the compact 
$F_{\delta_{i+1}}(Y^{\delta_{i+1}})$ it is precompact.
Let $w_1,...w_N\subset W_i$ be a maximal collection 
such that $B(w_i, r_i/2)$ are disjoint.   Then
\be
W_i \subset \bigcup_{n=1}^N B(w_n, r_i)
\ee
and
\be
\mathcal{H}^m(W_i) \le \sum_{n=1}^N \mathcal{H}^m(B(w_n, r_i)
\le   (1/4)^m \sum_{n=1}^N \mathcal{H}^m(B(w_n, r_i/4)).
\ee
But it is not hard to see examining (\ref{ball-in-good}) 
that $B(w_n, r_i/2)$ are the limits
of disjoint balls in $M_j$, so
\be
\sum_{n=1}^N \mathcal{H}^m(B(w_n, r_i/4)) \le \limsup_{j\to \infty} \mathcal{H}_m(M_j^{\delta_i})\le V_0.
\ee
So $W_i$ has Hausdorff dimension $m$ and
\be
\mathcal{H}_m(W_i) \le V_0.
\ee
Now
\be
Y=\bigcup_{i=1}^\infty W_i
\ee
so it has Hausdorff dimension $m$ and
\be
\mathcal{H}_m(Y) \le V_0.   
\ee

Now to see that $Y$ has positive density everywhere,
we must show
\be
\Theta_*(y, \mathcal{H}^m) =\liminf_{r\to 0} \frac{\mathcal{H}_m(B(y,r))}{r^m}>0.
\ee
For fixed $i\ge I_y$, we have
\be
\mathcal{H}_m(B(y, r)) \ge 
\mathcal{H}_m(B(y, r)\cap W_i).
\ee
Combining this with (\ref{Bishop-Gromov}) we have
\begin{eqnarray}
\Theta_*(y, \mathcal{H}^m) &=&\liminf_{r\to 0} \frac{\mathcal{H}_m(B(y,r)\cap W_i)}{r^m}\\
&\ge&\liminf_{r\to 0} \frac{\mathcal{H}_m(B(y,r_i)\cap W_i)}{r_i^m}\\
&\ge &\frac{\mathcal{H}_m(B(y,r_i)\cap W_i)}{r_i^m}>0.
\end{eqnarray}
\end{proof}

\begin{question} \label{q-ricci-unique}
Are the glued limit spaces of sequences as in Theorem~\ref{thm-ricci-glued-lim-measure} unique?
\end{question}

\begin{question} \label{q-ricci-rectifiable}
Are the glued limit spaces of sequences as in Theorem~\ref{thm-ricci-glued-lim-measure} countably $\mathcal{H}^m$ rectifiable?
\end{question}

\bibliographystyle{plain}
\bibliography{2012}

\end{document}